\documentclass[final]{dmtcs-episciences}

\usepackage[utf8]{inputenc}
\usepackage{subfigure}

\usepackage[round]{natbib}

\usepackage{amsmath,amsthm}

\newtheorem{thm}{Theorem}

\newtheorem{lem}[thm]{Lemma}
\newtheorem{prop}[thm]{Proposition}
\newtheorem{prob}[thm]{Problem}
\newtheorem{obs}[thm]{Observation}
\newtheorem{cor}[thm]{Corollary}
\theoremstyle{definition}
\newtheorem{defin}[thm]{Definition}

\newcommand{\rd}[1]{{#1}}
\newcommand{\rdrd}[1]{{#1}}

\newcommand{\IN}{\ensuremath{{\mathbb N}}}

\author{Stephan Dominique Andres\affiliationmark{1}
  \and Wai Lam Fong\affiliationmark{2}
}
\title{Line game-perfect graphs}
\affiliation{
  Institute of Mathematics and Computer Science, University of Greifswald, Germany\\
  Department of Mathematics, The University of Hong Kong, Hong Kong SAR, China}
\begin{document}
\publicationdata{vol. 26:2}{2024}{12}{10.46298/dmtcs.10971}{2023-02-17; 2023-02-17; 2024-02-15}{2024-05-24}
\maketitle
\begin{abstract}
\phantom{.}\\
The \emph{$[X,Y]$-edge colouring game} is played with a set
of $k\in\IN$ colours on a graph $G$ with initially
uncoloured edges by two players, Alice (A) and Bob (B).
The players move alternately. Player $X\in\{A,B\}$ has the first move.
$Y\in\{A,B,-\}$. If $Y\in\{A,B\}$, then only player $Y$ may skip any move,
otherwise skipping is not allowed for any player.
A move \rd{consists of} colouring an uncoloured edge with one of the $k$ colours
such that adjacent edges have distinct colours. When no more moves are possible,
the game ends. If every edge is coloured in the end, Alice wins; otherwise, 
 Bob wins.

The \emph{$[X,Y]$-game chromatic index} $\chi_{[X,Y]}'(G)$ 
is the smallest nonnegative integer $k$ such that Alice has a winning
strategy for the $[X,Y]$-edge colouring game played on $G$ with $k$ colours.
The graph $G$ is called \emph{line $[X,Y]$-perfect} if,
for any edge-induced subgraph $H$ of $G$,
\[\chi_{[X,Y]}'(H)=\omega(L(H)),\]
where $\omega(L(H))$ denotes the clique number of the line graph of~$H$.

For each of the six possibilities $(X,Y)\in\{A,B\}\times\{A,B,-\}$,
we characterise line $[X,Y]$-perfect graphs by forbidden {edge-induced}
subgraphs and by explicit structural descriptions, respectively.
\end{abstract}


\section{Introduction}

\rd{A subgraph $H$ of a graph $G$ is an \emph{induced subgraph} of $G$ 
if every edge of $G$ whose two end-vertices are in the vertex set of $H$ 
is also an edge of $H$. 
A subgraph $H$ of a graph $G$ is an \emph{edge-induced subgraph} of $G$ 
if the vertex set of $H$ consists of all end-vertices of edges of $H$. 
Equivalently, an edge-induced subgraph of $G$ is a subgraph of $G$ 
that contains no isolated vertices.}

The \emph{line graph} $L(G)$ of a graph $G=(V,E)$ is the
graph $(V',E')$ with $V'=E$ and where
the edge set $E'$ of $L(G)$ 
is the set of all unordered pairs $\{e_1,e_2\}$ of elements in $E$ 
such that $e_1$ and $e_2$ are adjacent as edges in~$G$.


Edge colouring of a graph $G$ is equivalent to vertex colouring
of its line graph~$L(G)$. Moreover, the colouring parameter for
edge colouring, the chromatic index $\chi'(G)$ of~$G$, equals
the chromatic number $\chi(L(G))$ of the line graph~$L(G)$.

In this sense, the well-established concept of perfectness for
vertex colouring has an interesting analog for edge colouring, namely
line perfectness, which was first defined by \cite{trotter}.
A graph $G$ is \emph{line perfect} if, for any {edge-induced} subgraph
$H$ of $G$,
\[\chi'(H)=\omega(L(H)),\]
where $\omega(L(H))$, the \emph{clique number} of the line graph of~$H$,
is the maximum number of mutually adjacent edges in~$H$.
Equivalently, as remarked by \cite{trotter}, 
a graph is line perfect if its line graph is a perfect graph.
The reason for this equivalence is the fact that, for any graph~$G$,
the set of the line graphs of the edge-induced subgraphs of~$G$ is the same
as the set of 
\rd{the induced} subgraphs of~$L(G)$.

It is known that line perfect graphs are perfect:

\begin{thm}[\cite{trotter}]\label{thm:trotter}
Line perfect graphs are perfect.
\end{thm}

\cite{trotter} gave a characterisation of line perfect graphs
by a set of forbidden edge-induced subgraphs:

\begin{thm}[\cite{trotter}]\label{thm:trotterchar}
A graph is line perfect if and only if it contains no odd cycles of
length at least 5 as edge-induced subgraphs.
\end{thm}

\cite{maffray} gave a complete characterisation of the
structure of line perfect graphs:

\begin{thm}[\cite{maffray}]\label{thm:maffray}
A graph $G$ is line perfect if and only if each of its blocks is either bipartite
or a complete graph $K_4$ on 4 vertices or a triangular book $K_{1,1,n}$ for
some positive integer~$n$. 
\end{thm}

Since bipartite graphs, $K_4$, and triangular books are perfect,
the powerful Theorem~\ref{thm:maffray} implies the result of 
Theorem~\ref{thm:trotter}.

In this paper we combine the idea of line perfect graphs with
\emph{graph colouring games}. For each such game, we define a notion of
\emph{line game-perfectness} and give a characterisation of the structure {of}
such line game-perfect graphs.
These characterisations include two
equivalent descriptions: a characterisation by forbidden edge-induced
subgraphs analog to the Theorem of Trotter (Theorem~\ref{thm:trotterchar}) 
and a characterisation by an explicit structural description
analog to the Theorem of Maffray 
(Theorem~\ref{thm:maffray}). 
 
\subsection{Vertex colouring games}

A \emph{vertex colouring game} is played 
with a set of
$k$ colours ($k\in\IN$) 
on a graph $G=(V,E)$ whose
vertices are initially uncoloured 
by two players, Alice~(A) and Bob~(B).
The players move alternately.
A move consists in colouring an uncoloured vertex with one of the
$k$ colours such that adjacent vertices receive distinct colours.
The game ends when such a move is not possible.
If every vertex is coloured in the end, Alice wins. Otherwise, i.e.,
in the case that an uncoloured vertex is adjacent to vertices of all
$k$ colours, Bob wins.

Such a \emph{graph colouring game}, defined by Brams, appeared in a 
mathematical games
column by \cite{gardner}. Later it was reinvented by 
\cite{bodlaender} who defined the \emph{game chromatic number}
$\chi_g(G)$ as the smallest nonnegative integer $k$ such that Alice has
a winning strategy for the vertex colouring game played on~$G$. Since Alice
always wins if $k\ge|V|$, the parameter is well-defined.

To be precise, two more rules have to be fixed to make the game well-defined.
Firstly, we have to fix the player $X\in\{A,B\}$ who moves first.
Secondly, we have to fix whether skipping (any) moves is allowed for some
player $Y\in\{A,B\}$ or not allowed (which we denote by $Y\in\{-\}$).
Thus we have six different games, one game for any of the pairs
\[(X,Y)\in\{A,B\}\times\{A,B,-\},\]
and we call such a game the $[X,Y]$-colouring game and denote its
game chromatic number by $\chi_{[X,Y]}(G)$.

The distinction of the six games is important when we discuss game-theoretic
analogs of perfect graphs, the \emph{game-perfect graphs}.

\subsection{Game-perfect graphs}

A graph $G$ is \emph{$[X,Y]$-perfect} (or \emph{game-perfect} for the
$[X,Y]$-colouring game) if, for any induced subgraph $H$ of $G$,
\[\chi_{[X,Y]}(H)=\omega(H),\]
where $\omega(H)$, the \emph{clique number} of $H$, is the maximum number
of mutually adjacent vertices in~$H$.

The concept of game-perfect graphs was introduced by 
\cite{andresphd,andresgperfect}. For four of the six games,
structural characterisations of game-perfect graphs by forbidden
induced subgraphs and by an explicit structural description are known.
The characterisation by forbidden induced subgraphs
of two of these characterisations will be extremely useful 
as basis for two of our main theorems in the following
sections:

\begin{figure}[htbp]
\begin{center}
\begin{minipage}[b]{0.23\textwidth}
\begin{center}
\includegraphics[scale=0.4]{P4}

$P_4$
\end{center}
\end{minipage}
\begin{minipage}[b]{0.23\textwidth}
\begin{center}
\includegraphics[scale=0.4]{C4}

$C_4$
\end{center}
\end{minipage}
\begin{minipage}[b]{0.23\textwidth}
\begin{center}
\includegraphics[scale=0.4]{trianglestar}

triangle star
\end{center}
\end{minipage}
\begin{minipage}[b]{0.23\textwidth}
\begin{center}
\includegraphics[scale=0.4]{XiGraph}

$\Xi$-graph
\end{center}
\end{minipage}

\bigskip

\begin{minipage}[b]{0.31\textwidth}
\begin{center}
\includegraphics[scale=0.4]{split3-star}

\includegraphics[scale=0.4]{split3-star}

two split 3-stars
\end{center}
\end{minipage}
\begin{minipage}[b]{0.31\textwidth}
\begin{center}
\includegraphics[scale=0.4]{doublefan}

\includegraphics[scale=0.4]{split3-star}

mixed graph
\end{center}
\end{minipage}
\begin{minipage}[b]{0.31\textwidth}
\begin{center}
\includegraphics[scale=0.4]{doublefan}

\includegraphics[scale=0.4]{doublefan}

two double fans
\end{center}
\end{minipage}
\end{center}
\caption{\label{fig:forbiddenAperf}Forbidden induced subgraphs for 
$[A,-]$-perfect graphs}
\end{figure}

\begin{figure}[htbp]
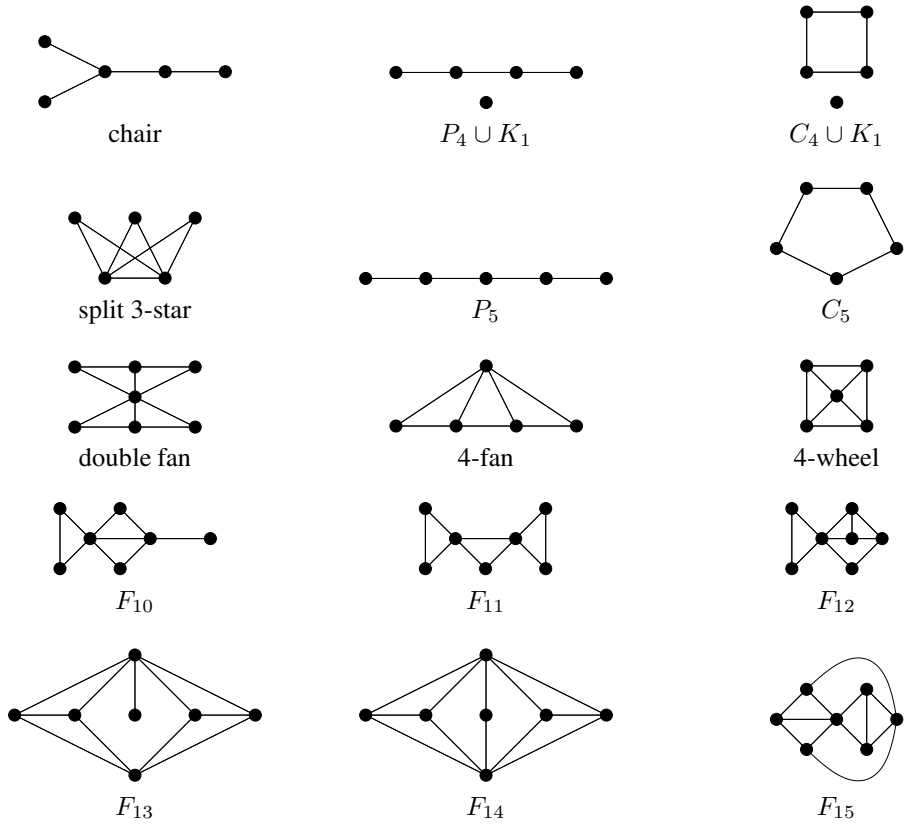

\begin{center}
\begin{minipage}[b]{0.31\textwidth}
\begin{center}
\includegraphics[scale=0.4]{chair}

chair
\end{center}
\end{minipage}
\begin{minipage}[b]{0.31\textwidth}
\begin{center}
\includegraphics[scale=0.4]{P4uK1}

$P_4\cup K_1$
\end{center}
\end{minipage}
\begin{minipage}[b]{0.31\textwidth}
\begin{center}
\includegraphics[scale=0.4]{C4uK1}

$C_4\cup K_1$
\end{center}
\end{minipage}

\bigskip

\begin{minipage}[b]{0.22\textwidth}
\begin{center}
\includegraphics[scale=0.4]{split3-star}

split 3-star
\end{center}
\end{minipage}
\begin{minipage}[b]{0.4\textwidth}
\begin{center}
\includegraphics[scale=0.4]{P5}

$P_5$
\end{center}
\end{minipage}
\begin{minipage}[b]{0.22\textwidth}
\begin{center}
\includegraphics[scale=0.4]{C5}

$C_5$
\end{center}
\end{minipage}

\bigskip

\begin{minipage}[b]{0.31\textwidth}
\begin{center}
\includegraphics[scale=0.4]{doublefan}

double fan
\end{center}
\end{minipage}
\begin{minipage}[b]{0.31\textwidth}
\begin{center}
\includegraphics[scale=0.4]{4-fan}

4-fan
\end{center}
\end{minipage}
\begin{minipage}[b]{0.31\textwidth}
\begin{center}
\includegraphics[scale=0.4]{4-wheel}

4-wheel
\end{center}
\end{minipage}

\bigskip

\begin{minipage}[b]{0.31\textwidth}
\begin{center}
\includegraphics[scale=0.4]{F10}

$F_{10}$
\end{center}
\end{minipage}
\begin{minipage}[b]{0.31\textwidth}
\begin{center}
\includegraphics[scale=0.4]{F11}

$F_{11}$
\end{center}
\end{minipage}
\begin{minipage}[b]{0.31\textwidth}
\begin{center}
\includegraphics[scale=0.4]{F12}

$F_{12}$
\end{center}
\end{minipage}

\bigskip

\begin{minipage}[b]{0.31\textwidth}
\begin{center}
\includegraphics[scale=0.4]{F13}

$F_{13}$
\end{center}
\end{minipage}
\begin{minipage}[b]{0.31\textwidth}
\begin{center}
\includegraphics[scale=0.4]{F14}

$F_{14}$
\end{center}
\end{minipage}
\begin{minipage}[b]{0.31\textwidth}
\begin{center}
\includegraphics[scale=0.4]{F15}

$F_{15}$
\end{center}
\end{minipage}
\end{center}
\caption{\label{fig:forbiddenBperf}Forbidden induced subgraphs for 
$[B,-]$-perfect graphs}
\end{figure}

\begin{thm}[\cite{char}]\label{thm:A}
A graph is $[A,-]$-perfect if and only if it contains none of the seven
graphs depicted in Figure~\ref{fig:forbiddenAperf} as an induced subgraph.
\end{thm}

\begin{thm}[\cite{lock}; \cite{andreslock}]\label{thm:B}
A graph is $[B,-]$-perfect if and only if it contains none of the fifteen
graphs depicted in Figure~\ref{fig:forbiddenBperf} as an induced subgraph.
\end{thm}

Furthermore, \cite{char} characterised the class of
$[B,B]$-perfect graphs by an explicit structural description. 
An \emph{ear animal} is a graph of the form
\[K_1\vee((K_a\vee(K_b\cup K_c))\cup K_{d_1}\cup K_{d_2}\cup\ldots\cup K_{d_k})\]
with $k\ge0$ and $a,b,c,d_1,d_2,\ldots,d_k\ge0$. 
Here, $G_1\cup G_2$ (resp., $G_1\vee G_2$) denotes the disjoint union 
(resp., the
join) of two graphs $G_1$ and $G_2$.

\begin{thm}[\cite{char}]\label{thm:BBcharE}
A graph is $[B,B]$-perfect if and only if 
each of its components is an ear animal.
\end{thm}

We will use this result in order to simplify the proof of one of our main
theorems (Theorem~\ref{thm:lineBB}).

\cite{lock} and \cite{andreslock} characterised the class
of $[B,-]$-perfect graphs by a (very large) explicit structural description.
From this description we will need only two partial results, given
in Proposition~\ref{prop:BEfive} and Proposition~\ref{prop:BEsix},
in order to simplify the proof of another one of our main theorems 
(Theorem~\ref{thm:lineB}).

\begin{figure}[htbp]
\begin{center}
\includegraphics[scale=0.4]{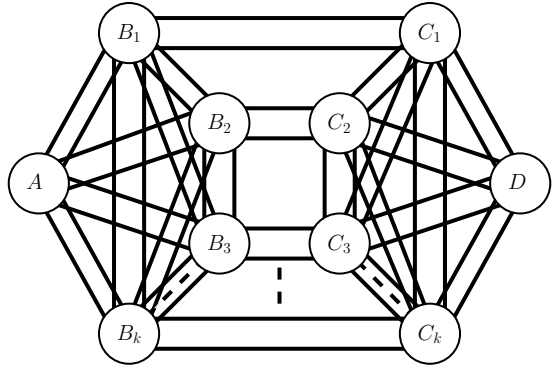}
\end{center}
\caption{\label{fig:expandedcocobi}An expanded cocobi}
\end{figure}

\begin{figure}[htbp]
\begin{center}
\includegraphics[scale=0.4]{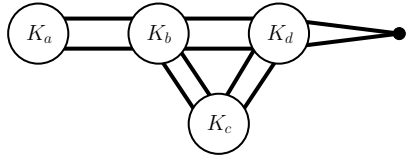}
\end{center}
\caption{\label{fig:expandedbull}An expanded bull}
\end{figure}

An \emph{expanded cocobi} 
is a graph of the type depicted in 
Figure~\ref{fig:expandedcocobi} with $k\ge1$, $a,d\ge0$, and 
$b_1,\ldots,b_k\ge1$, $c_1,\ldots,c_k\ge1$. 
\rd{An expanded cocobi consists of $2k$ nonempty cliques
$B_1,\ldots,B_k,C_1,\ldots,C_k$ 
and two
(possibly empty) cliques $A$ and $D$ with cardinalities
$|B_i|=b_i$, $|C_i|=c_i$, $|A|=a$, and $|D|=d$, for $1\le i\le k$,
respectively, such that between all pairs of cliques in
$\{A,B_1,\ldots,B_k\}$ 
there is a complete join,
between all pairs of cliques in 
$\{D,C_1,\ldots,C_k\}$
there is a complete join, and, for any $i\in\{1,\ldots,k\}$,
there is a complete join between the cliques $B_i$ and $C_i$
(and no further edges are allowed).
The dashed lines in Figure~\ref{fig:expandedcocobi} indicate
that $k$ may be arbitrarily large.}
An \emph{expanded bull}
is a graph of the type depicted in 
Figure~\ref{fig:expandedbull} with $a,b,d\ge1$ and $c\ge0$.
In both figures, pairs of lines indicate complete joins of the cliques.

\begin{prop}[\cite{lock}; \cite{andreslock}]\label{prop:BEfive}
Every expanded cocobi is $[B,-]$-perfect.
\end{prop}

\begin{prop}[\cite{lock}; \cite{andreslock}]\label{prop:BEsix}
Every expanded bull is $[B,-]$-perfect.
\end{prop}

\subsection{The edge colouring game}

An \emph{edge colouring game} has the same rules as a vertex colouring game,
except that the players have to colour uncoloured edges (instead of
uncoloured vertices) such that adjacent edges receive distinct colours.
And Alice wins if every edge (instead of vertex) is coloured in the end.

The \emph{game chromatic index} $\chi_g'(G)$ of $G$ is the smallest nonnegative
integer $k$ such that Alice has a winning strategy for the
edge colouring game played on $G$ with $k$ colours.

Thus, the edge colouring game on $G$ is equivalent to the vertex colouring game
on the line graph $L(G)$ of $G$; and the game chromatic index and the
game chromatic number are related by
\begin{equation}
\chi_g'(G)=\chi_g(L(G)).
\end{equation}

The edge colouring game was introduced by \cite{lamshiuxu}
and \cite{caizhu}.
Determining the maximum game chromatic index of some classes of graphs
has been {the} subject of several 
papers (cf.\ \cite{andresforest,andresetalwheel,bartnickigrytczuk,caizhu,%
channong,charpentierep,fongchan,fongchanarxiv,fongchannong,lamshiuxu})
as well as the game chromatic index {of} random 
graphs (cf.\ \cite{beveridgeetal,keusch})
as well as different types of edge-colouring based 
games (cf.\ \cite{boudonetal,dunnrelkdeg,dunnmorawskinordstrom}).

Originally, in the edge colouring game Alice moves first and skipping is
not allowed. However, we will distinguish here six variants as in the
vertex colouring game. Namely, in the 
\emph{$[X,Y]$-edge colouring game} player $X$ has the first move and
player $Y$ may skip moves (in case $Y\in\{A,B\}$) or none of the players
is allowed to skip (in case $Y\in\{-\})$. Its corresponding
game chromatic index is denoted by $\chi_{[X,Y]}'(G)$.

\subsection{Combining the setting: line game-perfect graphs}

A graph $G$ is \emph{line $[X,Y]$-perfect} (or \emph{line game-perfect} for the
$[X,Y]$-edge colouring game) if, for any edge-induced subgraph $H$ of $G$,
\[\chi_{[X,Y]}'(H)=\omega(L(H)).\]

In this paper, we characterise line game-perfect graphs
\begin{itemize}
\item for the games $[B,B]$, $[A,B]$, and $[A,-]$ (Theorem~\ref{thm:lineBB});
\item for game $[B,-]$ (Theorem~\ref{thm:lineB});
\item 
for game $[B,A]$ (Theorem~\ref{thm:lineBA});
\item 
for game $[A,A]$ (Theorem~\ref{thm:lineAA}).
\end{itemize}







\subsection{Main results}
\label{sec:main}


\begin{thm}\label{thm:lineAA}
The following are equivalent for a graph $G$.
\begin{itemize}
\item[(1)] $G$ is line $[A,A]$-perfect.
\item[(2)] 
None of the following configurations
(depicted in Figure~\ref{fig:fAA}) is an edge-induced subgraph of~$G$:
\rd{$P_6$, $C_5$, mini lobster $F_2$,
trigraph $F_3$,} or two 3-caterpillars $F_1\cup F_1$. 
\item[(3)] At most one component of $G$ is a full tree of type $E_1$ 
or a satellite of type $E_2$, 
and every other component is a single galaxy, a double galaxy, a candy, 
a star book, a diamond of flowers, or a tetrahedron of flowers
(described in Sections~\ref{subsec:BA} and~\ref{subsec:AA}).
\end{itemize}
\end{thm}

\pagebreak[3]

\begin{thm}\label{thm:lineBA}
The following are equivalent for a graph $G$.
\begin{itemize}
\item[(1)] $G$ is line $[B,A]$-perfect.
\item[(2)] 
None of the following configurations
(depicted in Figure~\ref{fig:fBA})
is an edge-induced subgraph of~$G$:
\rd{$P_6$, $C_5$, or 3-caterpillar~$F_1$.} 
\item[(3)] Every component of $G$ is a single galaxy, a double galaxy, a candy,
a star book, a diamond of flowers, or a tetrahedron of flowers
(described in Sections~\ref{subsec:B} and~\ref{subsec:BA}).
\end{itemize}
\end{thm}

\begin{thm}\label{thm:lineB}
The following are equivalent for a graph $G$.
\begin{itemize}
\item[(1)] $G$ is line $[B,-]$-perfect.
\item[(2)] 
None of the following configurations
(depicted in Figure~\ref{fig:fB})
is an edge-induced subgraph of~$G$:
\rd{$P_5\cup P_2$, $C_4\cup P_2$,
$P_6$, $C_5$, bull, diamond, or 3-caterpillar~$F_1$.} 
\item[(3)] Every component of $G$ is a double star, a vase of flowers,
or an isolated vertex, or $G$ contains exactly one nontrivial component
and this component is a double star, a vase of flowers, a candy,
a shooting star, a double vase, or an amaryllis
(described in Sections~\ref{subsec:BB} and~\ref{subsec:B}).
\end{itemize}
\end{thm}

\begin{thm}\label{thm:lineBB}
The following are equivalent for a graph $G$.
\begin{itemize}
\item[(1)] $G$ is line $[B,B]$-perfect.
\item[(2)] $G$ is line $[A,B]$-perfect.
\item[(3)] $G$ is line $[A,-]$-perfect.
\item[(4)] 
None of the following configurations
(depicted in Figure~\ref{fig:fBB})
is an edge-induced subgraph of~$G$:
\rd{$P_5$} or $C_4$. 
\item[(5)] Every component of $G$ is either a vase of flowers or
a double star or an isolated vertex (described in Section~\ref{subsec:BB}).
\end{itemize}
\end{thm}

\begin{figure}[htbp]
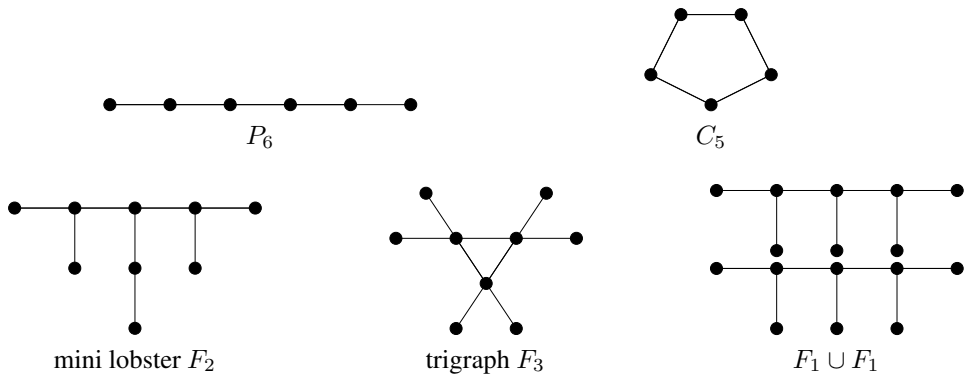

\begin{center}
\begin{minipage}[b]{0.4\textwidth}
\begin{center}
\includegraphics[scale=0.4]{P6}

$P_6$
\end{center}
\end{minipage}
\begin{minipage}[b]{0.4\textwidth}
\begin{center}
\includegraphics[scale=0.4]{C5}

$C_5$
\end{center}
\end{minipage}

\bigskip

\begin{minipage}[b]{0.31\textwidth}
\begin{center}
\includegraphics[scale=0.4]{FF2}

mini lobster $F_2$
\end{center}
\end{minipage}
\begin{minipage}[b]{0.31\textwidth}
\begin{center}
\includegraphics[scale=0.4]{FF3}

trigraph $F_3$
\end{center}
\end{minipage}
\begin{minipage}[b]{0.31\textwidth}
\begin{center}
\includegraphics[scale=0.4]{3-caterpillar}

\includegraphics[scale=0.4]{3-caterpillar}

$F_1\cup F_1$
\end{center}
\end{minipage}
\end{center}
\caption{\label{fig:fAA}Forbidden subgraphs for line $[A,A]$-perfect graphs}
\end{figure}

\begin{figure}[htbp]
\begin{center}
\begin{minipage}[b]{0.4\textwidth}
\begin{center}
\includegraphics[scale=0.4]{P6}

$P_6$
\end{center}
\end{minipage}
\begin{minipage}[b]{0.21\textwidth}
\begin{center}
\includegraphics[scale=0.4]{C5}

$C_5$
\end{center}
\end{minipage}
\begin{minipage}[b]{0.31\textwidth}
\begin{center}
\includegraphics[scale=0.4]{3-caterpillar}

3-caterpillar $F_1$
\end{center}
\end{minipage}
\end{center}
\caption{\label{fig:fBA}Forbidden subgraphs for line $[B,A]$-perfect graphs}
\end{figure}

\begin{figure}[htbp]
\begin{center}
\begin{minipage}[b]{0.4\textwidth}
\begin{center}
\includegraphics[scale=0.4]{P5uP2}

$P_5\cup P_2$
\end{center}
\end{minipage}
\begin{minipage}[b]{0.4\textwidth}
\begin{center}
\includegraphics[scale=0.4]{C4uP2}

$C_4\cup P_2$
\end{center}
\end{minipage}

\bigskip

\begin{minipage}[b]{0.4\textwidth}
\begin{center}
\includegraphics[scale=0.4]{P6}

$P_6$
\end{center}
\end{minipage}
\begin{minipage}[b]{0.4\textwidth}
\begin{center}
\includegraphics[scale=0.4]{C5}

$C_5$
\end{center}
\end{minipage}

\bigskip

\begin{minipage}[b]{0.31\textwidth}
\begin{center}
\includegraphics[scale=0.4]{bull}

bull
\end{center}
\end{minipage}
\begin{minipage}[b]{0.31\textwidth}
\begin{center}
\includegraphics[scale=0.4]{diamond}

diamond
\end{center}
\end{minipage}
\begin{minipage}[b]{0.31\textwidth}
\begin{center}
\includegraphics[scale=0.4]{3-caterpillar}

3-caterpillar $F_1$
\end{center}
\end{minipage}
\end{center}
\caption{\label{fig:fB}Forbidden subgraphs for line $[B,-]$-perfect graphs}
\end{figure}


\begin{figure}[htbp]
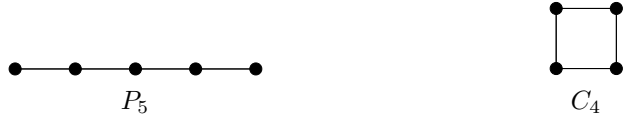

\begin{center}
\begin{minipage}[b]{0.4\textwidth}
\begin{center}
\includegraphics[scale=0.4]{P5}

$P_5$
\end{center}
\end{minipage}
\begin{minipage}[b]{0.4\textwidth}
\begin{center}
\includegraphics[scale=0.4]{C4}

$C_4$
\end{center}
\end{minipage}
\end{center}
\caption{\label{fig:fBB}Forbidden subgraphs for line game-perfect graphs for 
the games $[B,B]$, $[A,B]$, and $[A,-]$}
\end{figure}

\subsection{Idea of the proof}

In order to prove Theorem~\ref{thm:lineAA}, 
first, in Section~\ref{sec:forbidden}, we show that 
Bob has a winning strategy for the $[A,A]$-edge colouring game played
on each of the forbidden configurations,
which proves
the implication (1)$\Longrightarrow$(2); then, in Section~\ref{sec:permitted}, we show that 
Alice has a winning strategy for the $[A,A]$-edge colouring game played on each of
the permitted types from (3)
and, 
in Section~\ref{sec:proofs}, we show that the permitted types are hereditary
(i.e., every { edge-induced} subgraph of such a permitted type is in one of the permitted types),
which together proves
the \mbox{implication~(3)$\Longrightarrow$(1);} finally, in Section~\ref{sec:proofs}, using Theorem~\ref{thm:maffray} 
we prove the structural description (i.e.,
the permitted types are exactly those graphs that do not contain any of
the forbidden configurations as an edge-induced subgraph),
which settles {the} implication (2)$\Longrightarrow$(3).

The proofs of Theorems~\ref{thm:lineBA}, \ref{thm:lineB} and~\ref{thm:lineBB} 
have the same structure,
however the structural implication 
\rd{(i.e., (2)$\Longrightarrow$(3) in Theorem~\ref{thm:lineBA}
and Theorem~\ref{thm:lineB}, respectively, 
(4)$\Longrightarrow$(5) in Theorem~\ref{thm:lineBB})}
can be simplified by using the
structural result from Theorem~\ref{thm:lineAA}.

Furthermore, the other implications in Theorem~\ref{thm:lineB} (respectively,
Theorem~\ref{thm:lineBB}) can be obtained in an easy way by using
the structural results known 
for $[B,-]$-perfect (respectively, $[A,-]$-, $[A,B]$- and $[B,B]$-perfect)
graphs, namely Theorem~\ref{thm:B} and Propositions~\ref{prop:BEfive}
and~\ref{prop:BEsix} 
(respectively, Theorem~\ref{thm:A} and Theorem~\ref{thm:BBcharE}).
The method for this last simplification will be described in Section~\ref{sec:idea}.

\subsection{Structure of this paper}

In Section~\ref{sec:prelim} we give some basic definitions and fix notation.
A general idea to simplify the proofs concerning those games where
a structural characterisation for game-perfectness is known
is given in Section~\ref{sec:idea}.
In Section~\ref{sec:forbidden} we show that the forbidden configurations 
are not line game-perfect.
In Section~\ref{sec:permitted} we describe winning strategies for Alice
on the permitted types 
for the different games.
After these preparations,
Section~\ref{sec:proofs}
is devoted to the proofs
of the four main theorems. 
We conclude with an open question and a related problem
in Section~\ref{sec:final}.

\section{Preliminaries}
\label{sec:prelim}

We start by giving some definitions, easy observations and some results
from the literature that we will use.










\subsection{Notation}

All graphs considered in this paper are simple, i.e., they contain
neither loops nor multiple edges. Let $G$ be a graph. We denote by 

\begin{itemize}
\item
$\IN:=\{0,1,2,3,\ldots\}$ the set of nonnegative integers;
\item
$\Delta(G)$ the maximum degree of $G$;
\item
$\omega(G)$ the clique number of $G$;
\item
$L(G)$ the line graph of $G$;
\item
$\chi'(G)$ the chromatic index of $G$;
\item
$\chi(G)$ the chromatic number of $G$;
\item
$\chi_{[X,Y]}'(G)$ the game chromatic index w.r.t.\ edge game $[X,Y]$;
\item
$\chi_{[X,Y]}(G)$ the game chromatic number w.r.t.\ vertex game $[X,Y]$.
\end{itemize}

Let $m,n\in\IN$. By $P_n$ ($n\ge1$), $C_n$ ($n\ge3$), $K_n$, and $K_{m,n}$,
we denote the path, cycle and complete graph on $n$ vertices, and
the complete bipartite graph with partite sets of $m$ and $n$ vertices, 
respectively.

\begin{defin}
Let $G$ be a graph. An edge of $G$ is called \emph{unsafe} if it is
adjacent to at least $\omega(L(G))$ edges.
\end{defin}









\subsection{Basic observations}

The different vertex colouring games are related as follows:

\begin{obs}[\cite{andresgperfect}]\label{obs:comparechig}
For any graph $G$,
\begin{eqnarray*}
&\omega(G)\le\chi(G)\le\chi_{[A,A]}(G)\le\chi_{[A,-]}(G)\le\chi_{[A,B]}(G)\le\chi_{[B,B]}(G)&\\
&\omega(G)\le\chi(G)\le\chi_{[A,A]}(G)\le\chi_{[B,A]}(G)\le\chi_{[B,-]}(G)\le\chi_{[B,B]}(G)&
\end{eqnarray*}
\end{obs}

The same holds for the edge colouring games:

\begin{obs}[\cite{andresforest}]\label{obs:comparechigL}
For any graph $G$,
\begin{eqnarray*}
&\omega(L(G))\le\chi'(G)\le\chi_{[A,A]}'(G)\le\chi_{[A,-]}'(G)\le\chi_{[A,B]}'(G)\le\chi_{[B,B]}'(G)&\\
&\omega(L(G))\le\chi'(G)\le\chi_{[A,A]}'(G)\le\chi_{[B,A]}'(G)\le\chi_{[B,-]}'(G)\le\chi_{[B,B]}'(G)&
\end{eqnarray*}
\end{obs}

\subsection{Basic definitions and observations}

Recall that the \emph{line graph} $L(G)$ of a graph $G=(V,E)$ is the
graph $(E,E')$ where, for any $e_1,e_2\in E$, 
$e_1e_2$ is an edge in $L(G)$
(i.e., $e_1e_2\in E'$) 
if and only if the edges $e_1$ 
and $e_2$ are adjacent in~$G$. 

\begin{obs}\label{obs:aaa}
For any graph $G$ and any game $[X,Y]$ with $X\in\{A,B\}$ and $Y\in\{A,B,-\}$,
we have
\begin{eqnarray*}
\chi'(G)&=&\chi(L(G)),\\
\chi_{[X,Y]}'(G)&=&\chi_{[X,Y]}(L(G)).
\end{eqnarray*}
\end{obs}

Observation~\ref{obs:aaa} implies the two following
observations, which can be taken as alternative definitions of
line perfect graphs and line game-perfect graphs, respectively.



\begin{obs}[\cite{trotter}]\label{def:lineperfect}
A graph $G$ is \rd{\emph{line perfect} if} $L(G)$ is perfect, i.e., for any
vertex-induced subgraph $H'$ of $L(G)$,
\[\chi(H')=\omega(H').\]
\end{obs}



\begin{obs}\label{def:linegame}
A graph $G$ is \rd{\emph{line $[X,Y]$-perfect} if} $L(G)$ is $[X,Y]$-perfect, 
i.e., for any
vertex-induced subgraph $H'$ of $L(G)$,
\[\chi_{[X,Y]}(H')=\omega(H').\]
\end{obs}

Let ${\cal LP}[X,Y]$ be the class of line $[X,Y]$-perfect graphs
and ${\cal LP}$ be the class of line perfect graphs.
Then, by the definition of line perfect graphs
and line game-perfect graphs, Observation~\ref{obs:comparechigL} directly implies
(or, alternatively, by
Observation~\ref{def:lineperfect} and Observation~\ref{def:linegame}, 
Observation~\ref{obs:comparechig} directly implies) the following.

\begin{obs}\label{obs:compareclasses}
\begin{eqnarray*}
&{\cal LP}[B,B]\subseteq{\cal LP}[A,B]\subseteq{\cal LP}[A,-]\subseteq{\cal LP}[A,A]\subseteq{\cal LP}&\\
&{\cal LP}[B,B]\subseteq{\cal LP}[B,-]\subseteq{\cal LP}[B,A]\subseteq{\cal LP}[A,A]\subseteq{\cal LP}&\\
\end{eqnarray*}
\end{obs} 

In particular, Observation~\ref{obs:compareclasses} says that
every line $[X,Y]$-perfect graph is line perfect.

Using 
Theorem~\ref{thm:trotter} and 
Observation~\ref{obs:compareclasses} 
we get the following.

\begin{cor}
The classes of line $[X,Y]$-perfect graphs 
are subclasses of the class of perfect graphs.
\end{cor}

\begin{defin}\label{def:linenice}
A graph $G$ is \emph{line $[X,Y]$-nice} if
\[\chi_{[X,Y]}'(G)=\omega(L(G)),\]
i.e., if Alice has a winning strategy with $\omega(L(G))$ colours
for the $[X,Y]$-edge colouring game played on~$G$.
\end{defin}

The definition of line game-perfect graphs and Definition~\ref{def:linenice}
have an obvious relation, given in 
Observation~\ref{obs:niceperfect}.

\begin{obs}\label{obs:niceperfect}
A graph $G$ is line $[X,Y]$-perfect if and only if each of its
{edge-induced} subgraphs is {line} $[X,Y]$-nice.
\end{obs}

\pagebreak[3]

\subsection{Characterisations of line graphs}

The following well-known theorem will be used in our proofs.

\begin{thm}[\cite{whitney}]\label{thm:whitney}
Two connected nonempty graphs $G_1$ and $G_2$ are isomorphic if and only if their
line graphs $L(G_1)$ and $L(G_2)$ are isomorphic, with the single
exception of the two graphs $K_3$ and $K_{1,3}$, which have the
same line graph 
\[L(K_3)=L(K_{1,3})=K_3.\]
\end{thm}


\begin{cor}\label{cor:whitney}
For any graph $G$,
\[\rd{\omega(L(G))=\left\{\begin{array}{ll}
3&\text{if}\ \Delta(G)=2\ \text{and}\ G\ \text{contains\ a\ component}\ K_3\\
0&\text{if}\ G\ \text{is empty}\\
\Delta(G)&\text{otherwise} 
\end{array}
\right.}\]
\end{cor}

\begin{proof}[of Corollary~\ref{cor:whitney}]
Let $G$ be nonempty. By Theorem~\ref{thm:whitney}, cliques $K_n$ with $n\ge1$ in $L(G)$ originate only from
edge-induced stars $K_{1,n}$ or the $K_3$ in~$G$.
Thus, either a star of maximum degree $\Delta(G)$ leads to the value of
$\omega(L(G))$ or the three mutually adjacent edges in a $K_3$. 
\end{proof}

Using Observation~\ref{obs:aaa} and Corollary~\ref{cor:whitney}, we may
reformulate 
Definition~\ref{def:linenice}
for a graph $G$ with maximum degree $\Delta(G)\ge3$ as follows.
%

\begin{obs}
A graph $G$ with $\Delta(G)\ge3$ is \emph{line $[X,Y]$-nice} if
\[\chi_{[X,Y]}'(G)=\Delta(G).\]
\end{obs}

The precondition ``nonempty'' in Theorem~\ref{thm:whitney} is essential
as the line graphs of an isolated vertex $K_1$ and an empty graph $K_0$
are both the empty graph. Therefore, considering line graphs, it is
convenient to exclude isolated vertices, which motivates the following
definition.

\begin{defin}
A graph is \emph{iso-free} if it has no isolated vertices.
\end{defin}

\cite{beineke} gave a characterisation of line graphs
by forbidden induced subgraphs, which will be used in our proofs.

\begin{thm}[\cite{beineke}]\label{thm:beineke}
A graph $G$ is a line graph if and only if it contains none
of the nine graphs $N_1,\ldots,N_9$ from Figure~\ref{fig:lineforbid} 
as an induced subgraph. 
\end{thm}


\begin{figure}[htbp]
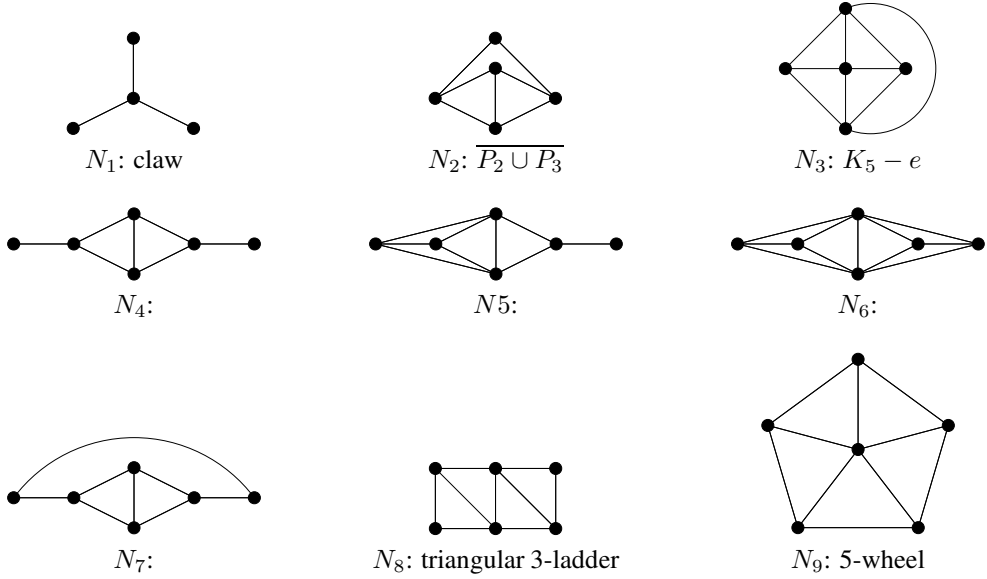

\begin{center}
\begin{minipage}[b]{0.32\textwidth}
\begin{center}
\includegraphics[scale=0.4]{N1}

$N_1$: claw
\end{center}
\end{minipage}
\begin{minipage}[b]{0.32\textwidth}
\begin{center}
\includegraphics[scale=0.4]{N2}

$N_2$: $\overline{P_2\cup P_3}$
\end{center}
\end{minipage}
\begin{minipage}[b]{0.32\textwidth}
\begin{center}
\includegraphics[scale=0.4]{N3}

$N_3$: $K_5-e$
\end{center}
\end{minipage}

\bigskip

\begin{minipage}[b]{0.32\textwidth}
\begin{center}
\includegraphics[scale=0.4]{N4}

$N_4$:
\end{center}
\end{minipage}
\begin{minipage}[b]{0.32\textwidth}
\begin{center}
\includegraphics[scale=0.4]{N5}

$N5$:
\end{center}
\end{minipage}
\begin{minipage}[b]{0.32\textwidth}
\begin{center}
\includegraphics[scale=0.4]{N6}

$N_6$:
\end{center}
\end{minipage}

\bigskip

\begin{minipage}[b]{0.32\textwidth}
\begin{center}
\includegraphics[scale=0.4]{N7}

$N_7$:
\end{center}
\end{minipage}
\begin{minipage}[b]{0.32\textwidth}
\begin{center}
\includegraphics[scale=0.4]{N8}

$N_8$: triangular 3-ladder
\end{center}
\end{minipage}
\begin{minipage}[b]{0.32\textwidth}
\begin{center}
\includegraphics[scale=0.4]{N9}

$N_9$: 5-wheel
\end{center}
\end{minipage}

\end{center}
\caption{\label{fig:lineforbid}The nine forbidden induced subgraphs for line graphs}
\end{figure}

\section{{\rd{A plan for}} characterising line game-perfect graphs}
\label{sec:idea}

Let $g=[X,Y]$ with $(X,Y)\in\{A,B\}\times\{A,B,-\}$.

\begin{itemize}
\item
For some games $g$, we have a characterisation of game-perfect graphs
by a list ${\cal F}_g$ of forbidden induced subgraphs.
\item
Let red$({\cal F}_g)$ be 
\rd{the maximal subset of ${\cal F}_g$ such that no graph
in red$({\cal F}_g)$ contains an induced $N_1,\ldots,N_9$.}
\item
Let $L^{-1}({\cal F}_g)$ the set of all iso-free graphs $H$ 
such that $L(H)\in{\cal F}_g$. 
\item
Let $L^{-1}({\rm red}({\cal F}_g))$ (which we call the \emph{reduced list}) be the set of all iso-free graphs $H$ 
such that $L(H)\in{\rm red}({\cal F}_g)$. 
\end{itemize}


Then we have following fundamental theorem (which
gives us the basic idea how to characterise game-perfect line graphs).


\begin{thm}\label{thm:grund}
The following are equivalent for a graph $G$ and a game $g=[X,Y]$ with
$X\in\{A,B\}$ and $Y\in\{A,B,-\}$.
\begin{itemize}
\item[(1)] $G$ is line $g$-perfect.
\item[(2)] No graph in $L^{-1}({\cal F}_g)$ is an edge-induced subgraph of $G$.
\item[(3)] No graph in $L^{-1}({\rm red}({\cal F}_g))$ is an edge-induced subgraph of $G$.
\end{itemize}
\end{thm}


\begin{proof}
(1)$\Longrightarrow$(2): 
Let $G$ be a line $g$-perfect graph.
Then $U:=L(G)$ is $g$-game-perfect.
By the characterisation of (vertex) $g$-perfectness 
(e.g., by Theorem~\ref{thm:A} or Theorem~\ref{thm:B}) $U$ does
not contain a graph from the list ${\cal F}_g$ as an induced subgraph.
That means that $G$ 
does not contain a graph from the
list $L^{-1}({\cal F}_g)$ as an {edge-induced} subgraph.

(2)$\Longrightarrow$(1): 
Let \rd{$G$ not} contain any graph from the list $L^{-1}({\cal F}_g)$ as
an {edge-induced} subgraph. Then $U:=L(G)$ does not contain
any graph from the list ${\cal F}_g$ as an induced subgraph.
By the characterisation of $g$-perfect graphs
(e.g., by Theorem~\ref{thm:A} or Theorem~\ref{thm:B}), this means that
$U$ is $g$-game-perfect. By definition this means that $G$ is
{line $g$-perfect}.

(2)$\Longrightarrow$(3): holds trivially. 

(3)$\Longrightarrow$(2): 
Let $H\in L^{-1}({\cal F}_g)$.
Then there is a graph $F\in{\cal F}_g$ with $F=L(H)$. Since
$F$ is a line graph, by Theorem~\ref{thm:beineke} $F$ does not contain any of the
forbidden configurations $N_1,\ldots,N_9$ as an induced subgraph.
Thus $F\in{\rm red}({\cal F}_g)$. Thus 
\[H\in L^{-1}({\rm red}({\cal F}_g)).\]
%
\end{proof}

\paragraph{General strategy to characterise $g$-game line perfect graphs.}
We do the following steps
\begin{enumerate}
\item Determine ${\cal F}_g$ (known from literature for 4 games):
\item Determine red$({\cal F}_g)$.
\item Determine $L^{-1}({\rm red}({\cal F}_g))$.
\item Characterise the class of graph which do not contain any member
of $L^{-1}({\rm red}({\cal F}_g))$
as an edge-induced subgraph.
\end{enumerate}

The above idea is used for the proofs of Theorem~\ref{thm:lineBB}
and Theorem~\ref{thm:lineB}. 
In Theorem~\ref{thm:lineBA} 
and Theorem~\ref{thm:lineAA}
we use different methods since, for the games
$[B,A]$ and $[A,A]$, no explicit characterisation
of the game-perfect graphs is known.












\section{The forbidden 
subgraphs}
\label{sec:forbidden}

\subsection{Forbidden 
subgraphs for game $[A,-]$}
\label{subsec:forbiddenA}

The following \rd{basic lemma is} implied by Theorem~\ref{thm:whitney}.

\begin{lem}\label{lem:onlyPfour}
\rd{We have:
\begin{itemize}
\item
$P_5$ is the only iso-free graph whose line graph is~$P_4$.
%
\item
$C_4$ is the only iso-free graph whose line graph is~$C_4$.
\end{itemize}}
\end{lem}

From the \rd{lemma} above we conclude the following.

\begin{prop}\label{prop:lineA}
A graph is line $[A,-]$-perfect if and only if it contains no $P_5$ or $C_4$ as
an edge-induced subgraph.
\end{prop}

\begin{proof}
Let $G$ be a graph.
By Theorem~\ref{thm:grund}, the graph $G$ is line $[A,-]$-perfect if and only if 
no graph in the reduced list
$L^{-1}({\rm red}({\cal F}_{[A,-]}))$
is an {edge-induced} subgraph of~$G$. Thus, it is sufficient to prove
that
\[L^{-1}({\rm red}({\cal F}_{[A,-]}))=\{P_5,C_4\}.\]

By 
Theorem~\ref{thm:A}
we know
that ${\cal F}_{[A,-]}$
consists of the seven graphs depicted in Figure~\ref{fig:forbiddenAperf}.

Since $N_1$ is an induced subgraph of the split 3-star and of the double
fan, it is an induced subgraph of the triangle star, the $\Xi$-graph,
the two split 3-stars, the two double fans, and the mixed graph. Thus we have
\[{\rm red}({\cal F}_{[A,-]})=\{P_4,C_4\}\]
Since, by Lemma~\ref{lem:onlyPfour}, the only iso-free graph whose line graph 
is $P_4$ is the $P_5$,
\rd{and the}
only iso-free 
graph whose line graph is $C_4$ is the $C_4$, we have
\[L^{-1}({\rm red}({\cal F}_{[A,-]}))=\{P_5,C_4\}.\]
\end{proof}

\subsection{Forbidden 
subgraphs for game $[B,-]$}
\label{subsec:forbiddenB}

%
%
%
%
%
%
%


\rd{We obtain the following characterisation of line $[B,-]$-perfect graphs.}

\begin{prop}\label{prop:lineB}
A graph is line \rd{$[B,-]$-perfect if} and only if it
contains no $P_5\cup P_2$, $C_4\cup P_2$, $P_6$, $C_5$, bull, diamond,
or 3-caterpillar $F_1$ as
an edge-induced subgraph.
\end{prop}

\begin{proof}
Let $G$ be a graph .
By Theorem~\ref{thm:grund}, $G$ is line $[B,-]$-perfect if and only if 
no graph in the reduced list
$L^{-1}({\rm red}({\cal F}_{[B,-]}))$
is an edge-induced subgraph of~$G$. 
Thus, to finish the proof it is sufficient to show that the reduced list $L^{-1}({\rm red}({\cal F}_{[B,-]}))$
consists of the forbidden graphs mentioned in the proposition.

By 
Theorem~\ref{thm:B}
we know
that ${\cal F}_{[B,-]}$
consists of the 15 graphs depicted in Figure~\ref{fig:forbiddenBperf}.


Since $N_1$ is an induced subgraph of the chair, the split 3-star,
the double fan, the $F_{10}$, 
the $F_{12}$,
the $F_{13}$,
the $F_{14}$, and
the $F_{15}$,
we have
\[{\rm red}({\cal F}_{[B,-]})=\{P_4\cup K_1,C_4\cup K_1,P_5,C_5,\text{4-fan},
\text{4-wheel},F_{11}\}.\]

\rd{Furthermore, the following observations are implied by Theorem~\ref{thm:whitney}.

\begin{itemize}
\item
$P_5\cup P_2$ is the only iso-free graph whose line graph is $P_4\cup K_1$.
\item
$C_4\cup P_2$ is the only iso-free graph whose line graph is $C_4\cup K_1$.
\item
$P_6$ is the only iso-free graph whose line graph is $P_5$.
\item
$C_5$ is the only iso-free graph whose line graph is $C_5$.
\item
The bull is the only iso-free graph whose line graph is the 4-fan.
\item
The diamond is the only iso-free graph whose line graph is the 4-wheel.
\item
The 3-caterpillar $F_1$ is the only iso-free graph whose line graph is the 
graph $F_{11}$ 
depicted in Figure~\ref{fig:forbiddenBperf}.
\end{itemize}

By these observations}
we conclude that
\[L^{-1}({\rm red}({\cal F}_{[B,-]}))=\{P_5\cup P_2,C_4\cup P_2,P_6,C_5,\text{bull},\text{diamond},\text{3-caterpillar}\}.\]
\end{proof}

\subsection{Forbidden 
subgraphs for game $[B,A]$}
\label{subsec:forbiddenBA}


The next lemma is an auxiliary result that will be used to simplify some
strategies in the forthcoming 
Lemma~\ref{notBAniceFone},
Lemma~\ref{notAAniceFtwo},
and
Lemma~\ref{notAAnicetwoFone}.

\begin{figure}[htbp]
\begin{center}
\includegraphics[scale=0.4]{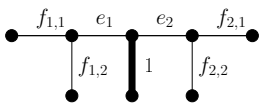}
\end{center}
\caption{\label{fig:cFone}The precoloured configuration $F_1^1$: 
\rd{the middle edge is precoloured by colour~1.}}
\end{figure}

\begin{lem}\label{lem:auxFoneone}
Bob has a winning strategy with 3 colours in the $[X,A]$-edge colouring
game on the precoloured graph $F_1^1$ 
\rd{where the middle edge is precoloured by colour~1}
(depicted in Figure~\ref{fig:cFone})
if it is Alice's turn.
\end{lem}

\begin{proof}
Assume the colour set is $\{1,2,3\}$.
By symmetry, Alice has four possible types of moves: she can colour $e_1$ with
colour~2, she can colour $f_{1,1}$ with colour~2 or with colour~1, or she
can miss her turn.
\begin{itemize}
\item If Alice colours $e_1$ with colour~2, Bob colours $f_{2,2}$ with colour~3.
Then there is no feasible colour for~$e_2$ any more.
\item If Alice colours $f_{1,1}$ with colour~2, Bob colours $f_{1,2}$ with colour~3.  
Then there is no feasible colour for~$e_1$ any more.
\item If Alice colours $f_{1,1}$ with colour~1 or skips, then Bob colours
$f_{2,1}$ with colour~2. If Bob will be able to colour 
$f_{2,2}$ or $e_1$ with colour~3 in his next move, he will win,
since there would be no feasible colour for $e_2$ in both cases. 
In order to avoid both
threats simultaneously, Alice has only one possibility: 
she must colour $e_2$ with
colour~3.
But then 
Bob colours
$f_{1,2}$ with colour~2, 
so that there is no feasible colour for~$e_1$ any more.
\end{itemize}
Thus, in any case, Bob will win.
\end{proof}

\begin{lem}\label{lem:BAnicecaterpillar}\label{notBAniceFone}
The 3-caterpillar is not line $[B,A]$-nice.
\end{lem}

\begin{proof}
Bob's winning strategy with 3 colours on the 3-caterpillar is to colour the 
central leaf edge $e_0$ (see Figure~\ref{fig:auxFone}) with colour 1 in his first move. 
Then we have the situation from Figure~\ref{fig:cFone}. Thus Bob wins by Lemma~\ref{lem:auxFoneone}.
\end{proof}

\subsection{Forbidden 
subgraphs for game $[A,A]$}
\label{subsec:forbiddenAA}

The first two of the following lemmata are trivial. We list them
for the sake of completeness.

\begin{lem}\label{lem:BAnicePsix}\label{notAAnicePsix}
The graph $P_6$ is not line $[A,A]$-nice.
\end{lem}

\begin{proof}
We describe a winning strategy with 2 colours for Bob in the $[A,A]$-edge colouring game
on the path $P_6$.

If Alice colours an edge in her first move, then Bob colours an edge
at distance 1 with the other colour, and the mutual adjacent edge of these
two edges cannot be coloured with any of the two colours.

Otherwise, if Alice misses her first turn, then
Bob's winning strategy with 2 colours on the $P_6$ is to colour the central
edge $e$ in colour 1 in his first move and one of the leaf edges with colour 2 in
his second move. Bob can colour at least one such leaf edge $f$, since
Alice has coloured at most one edge.
Then the mutual adjacent edge of the two edges $e$ and $f$ cannot be
coloured, and Bob wins. 

Thus the $P_6$ is not line $[A,A]$-nice.
\end{proof}


\begin{lem}\label{lem:AAniceCfive}\label{notAAniceCfive}
The graph $C_5$ is not line $[A,A]$-nice.
\end{lem}

\begin{proof}
The $C_5$ (is the smallest graph that) is not line perfect, 
thus it is not line $[A,A]$-nice. 
\end{proof}

\begin{figure}[htbp]
\begin{center}
\includegraphics[scale=0.4]{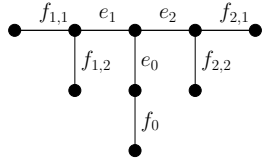}
\end{center}
\caption{\label{fig:auxiliary}The mini lobster $F_2$}
\end{figure}

\begin{lem}\label{lem:AAniceFtwo}\label{notAAniceFtwo}
The mini lobster $F_2$ is not line $[A,A]$-nice.
\end{lem}

\begin{proof}
By exhausting all possible first moves of Alice
we show that Bob has a winning strategy with 3 colours
for the $[A,A]$-edge colouring game played on~$F_2$.
We use the edge labels from Figure~\ref{fig:auxiliary}. 
By symmetry, it is sufficient to consider the cases that
Alice colours 
\rd{$f_0$, $e_0$, $f_{1,1}$, or $e_2$} 
in her first move, or
skips her first turn.
\begin{itemize}
\item If Alice skips or colours $f_0$, then Bob colours $e_0$ and thus creates
a configuration $F_1^1$. 
In the same way, if Alice colours $e_0$, then Bob colours $f_0$ and
creates a configuration $F_1^1$. 
In both cases, by Lemma~\ref{lem:auxFoneone}, Bob will win.
(Note that in case Alice colours $f_0$ in a forthcoming move, this may be considered
as skipping her turn in the situation of Lemma~\ref{lem:auxFoneone}.)
\item If Alice colours $f_{1,1}$ (w.l.o.g.\ with colour~1), then Bob colours
$e_2$ with colour~2. On the other hand, if Alice colours $e_2$ (w.l.o.g.\ with
colour 2), then Bob colours $f_{1,1}$ with colour~1.
In both cases, if Bob will be able to colour $f_{1,2}$ or $e_0$ with
colour~3 in his next move, then he would win as there would be no feasible
colour for $e_1$ any more. The only possibility for Alice 
to avoid both threats simultaneously is to colour $e_1$ with colour~3.
But then Bob colours $f_0$ with colour~1, so that there is no feasible colour
for $e_0$ any more.
\end{itemize}
In any case, Bob wins.
\end{proof}

\begin{figure}[htbp]
\begin{center}
\includegraphics[scale=0.4]{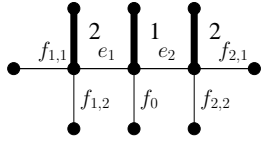}
\end{center}
\caption{\label{fig:cFthree}The precoloured configuration $F_3^1$}
\end{figure}

The next lemma is an auxiliary result that will be used to simplify some
strategies in the forthcoming Lemma~\ref{notAAniceFthree}.

\begin{lem}\label{lem:auxFthree}
Bob has a winning strategy with 4 colours in the $[X,A]$-edge colouring
game on the precoloured graph $F_3^1$ (depicted in Figure~\ref{fig:cFthree})
if it is Alice's turn.
\end{lem}

\begin{proof}
Assume the colour set is $\{1,2,3,4\}$.
By symmetry, Alice has six possible types of moves: she can miss her turn, colour $f_{1,1}$ with
colour~1 or with colour~3, she can colour $e_2$ with colour~3, or she can colour $f_{0}$ with colour~2 or with colour~3.
\begin{itemize}
\item If Alice misses her turn or colours $f_{1,1}$ (with colour~1 or colour~3), then Bob colours $f_0$ with
colour~3. After that only colour~4 is feasible for $e_1$ and $e_2$, which can be used only for one the these edges.
\item If Alice colours $e_2$ or $f_0$ with colour~3, Bob colours $f_{1,1}$ with colour~4. Then there is no
feasible colour for~$e_1$.
\item If Alice colours $f_{0}$ with colour~2, Bob colours
$f_{1,1}$ with colour~3. If Bob will be able to colour 
$f_{1,2}$ or $e_2$ with colour~4 in his next move, he will win,
since there would be no feasible colour for $e_1$ in both cases. 
In order to avoid both
threats simultaneously, Alice has only one possibility: 
she must colour $e_1$ with
colour~4.
But then 
Bob colours
$f_{2,1}$ with \rd{colour~3}, 
so that there is no feasible colour for~$e_2$ any more.
\end{itemize}
Thus, in any case, Bob will win.
\end{proof}

\begin{figure}[htbp]
\begin{center}
\includegraphics[scale=0.4]{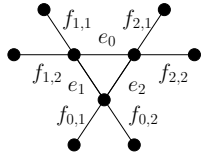}
\end{center}
\caption{\label{fig:auxiliaryT}The trigraph $F_3$}
\end{figure}

\begin{lem}\label{lem:AAniceFthree}\label{notAAniceFthree}
The trigraph $F_3$ is not line $[A,A]$-nice.
\end{lem}

\begin{proof}
We describe a winning strategy for Bob on $F_3$ with colour set $\{1,2,3,4\}$.
We use the edge labels from Figure~\ref{fig:auxiliaryT}.
Due to symmetry, Alice has three possibilities for her first move: She can
colour $e_0$ or $f_{0,1}$ with colour~1 or miss her turn.

If Alice colours $e_0$ with colour~1, then Bob colours $f_{0,1}$ with colour~2.
And, vice-versa, if Alice colours $f_{0,1}$ with colour~1, then Bob colours
$e_0$ with colour~2. In both cases (by considering the coloured edge~$e_0$ broken
in two parts) we get the precoloured configuration $F_3^1$ from Figure~\ref{fig:cFthree}.
By Lemma~\ref{lem:auxFthree}, Bob wins.

If Alice misses her turn, then Bob colours $f_{0,1}$ with colour~1.
Now, by symmetry, Alice has seven possibilities for her second move:
she may miss her turn, colour $e_0$ with colour~1 or with colour~2, colour $f_{0,2}$ with colour~2,
colour $e_1$ with colour~2, or colour $f_{1,1}$ with colour~2 or with colour~1. We make a case distinction
according to Alice's second move.
\begin{itemize}
\item 
If Alice misses her turn, then Bob colours $e_0$ with colour~2. Again,
we have a situation as in configuration~$F_3^1$, and, by Lemma~\ref{lem:auxFthree},
Bob wins.
\item
If Alice colours $e_0$ with colour~1, then this colour is forbidden for every uncoloured edge.
The uncoloured edges form a 3-caterpillar $F_1$ and the game is reduced to the $[B,A]$-edge colouring
game with three colours (2,3,4) on~$F_1$. 
Therefore, Bob wins by Lemma~\ref{notBAniceFone}.
\item
If Alice colours $e_0$ with colour~2, Bob colours $f_{0,2}$ with colour~3. And, vice-versa, if
Alice colours $f_{0,2}$ with colour~2, Bob colours $e_0$ with colour~3. In both cases
the only feasible colour remaining for $e_1$ and $e_2$ is colour~4, which cannot be assigned to both edges.
\item
If Alice colours $e_1$ with colour~2, then Bob colours $f_{2,1}$ with colour~3. If Bob will be able to
colour $f_{0,2}$ or $f_{2,2}$ with colour~4 in his next move, he will win as $e_2$ has no feasible colour.
The only possible move of Alice to avoid both threats is to colour $e_2$ with colour~4.
But then Bob colours $f_{1,1}$ with colour~1 and there is no feasible colour for $e_0$ any more.
\item
If Alice colours $f_{1,1}$ with colour~2, then Bob colours $f_{1,2}$ with colour~3. If Bob will be able to
colour $e_0$ or $f_{0,2}$ with colour~4 in his next move, he will win as $e_1$ has no feasible colour.
The only possible move of Alice to avoid both threats is to colour $e_1$ with colour~4.
But then Bob colours $f_{2,1}$ with colour~1 and there is no feasible colour for $e_0$ any more.
\item
If Alice colours $f_{1,1}$ with colour~1, then Bob colours $e_1$ with colour~2. If Bob will be able to
colour $f_{2,1}$ or $f_{2,2}$ with colour~3 in his next move, he will win as the only feasible colour
for the edges $e_0$ and $e_2$ will be colour~4, which can not be used for both edges.
The only possible moves of Alice to avoid both threats are to colour $e_0$ or $e_2$ with colour~3.
By symmetry, we may assume that Alice colours $e_0$ with colour~3.
But then Bob colours $f_{2,1}$ with colour~4 and there is no feasible colour for $e_2$ any more.
\end{itemize} 

Thus, in any case, Bob wins.
\end{proof}

\begin{figure}[htbp]
\begin{center}
\includegraphics[scale=0.4]{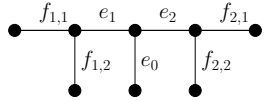}
\end{center}
\caption{\label{fig:auxFone}The 3-caterpillar $F_1$}
\end{figure}

\begin{lem}\label{lem:AAniceFoneFone}\label{notAAnicetwoFone}
The graph $F_1\cup F_1$ is not line $[A,A]$-nice.
\end{lem}

\begin{proof} 
We describe a winning strategy for Bob with 3 colours in the
$[A,A]$-edge colouring game played on a graph $F_1\cup F_1$ 
consisting of two components that are 3-caterpillars.
No matter what Alice does in her first move, Bob is able
to choose a component where every edge is uncoloured.
Bob colours the edge $e_0$ (see Figure~\ref{fig:auxFone}) in
this component and then, by playing always in this component,
he has a winning strategy by Lemma~\ref{lem:auxFoneone}.
\end{proof}











\section{The permitted types}
\label{sec:permitted}

\subsection{Permitted for game $[B,B]$}
\label{subsec:BB}

We start with some definitions. 

\begin{defin}[vase of flowers]
Let $n\in\IN$.
A \emph{vase of $n$ flowers} is a graph consisting of a $C_3$
and $n$ vertices of degree 1 that are adjacent with the same
vertex of the $C_3$, i.e., the graph has the vertex set
\[\{w_1,w_2,w_3,v_1,v_2,\ldots,v_n\}\]
and the edge set
\[\{w_1w_2,w_1w_3,w_2w_3\}\cup\{w_1v_i\mid 1\le i\le n\}.\]
A \emph{vase of flowers} is a vase of $n$ flowers for some $n\in\IN$.
A vase of 1 flower is often called a \emph{paw}, and a vase of
2 flowers is sometimes called a \emph{cricket}.
\end{defin}

\begin{defin}[double star]
Let $m,n\in\IN$.
An \emph{$(m,n)$-double star} is a graph formed by an $m$-star and
an $n$-star by connecting the \rd{centres} of the stars by an additional edge,
i.e., the graph has the vertex set
\[\{w_1,w_2,u_1,u_2,\ldots,u_m,v_1,v_2,\ldots,v_n\}\]
an the edge set
\[\{w_1w_2\}\cup\{w_1u_i\mid 1\le i\le m\}\cup\{w_2v_i\mid 1\le i\le n\}.\]
A \emph{double star} is an $(m,n)$-double star for some $m,n\in\IN$.
The $(0,0)$-double star is the complete graph $K_2$.
\end{defin}

\begin{figure}[htbp]
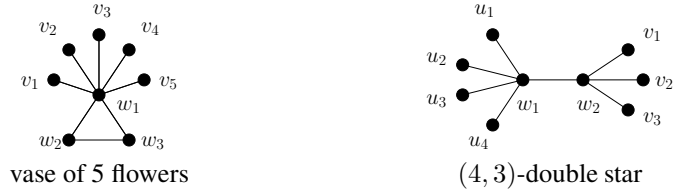

\begin{center}
\begin{minipage}[b]{0.4\textwidth}
\begin{center}
\includegraphics[scale=0.4]{Xvaseofflowers}

vase of 5 flowers
\end{center}
\end{minipage}
\begin{minipage}[b]{0.4\textwidth}
\begin{center}
\includegraphics[scale=0.4]{Xdoublestar}

$(4,3)$-double star
\end{center}
\end{minipage}
\end{center}
\caption{\label{deffigA}A vase of flowers and a double star}
\end{figure}

In order to prove the theorem, we start with some lemmata.


\begin{lem}\label{lem:AwinsBBB}
Graphs whose components are vases of flowers, double stars, or isolated vertices
are line $[B,B]$-perfect.
\end{lem}

\begin{proof}
Let $G$ be a graph whose components are vases of flowers, double stars, or
isolated vertices. 
The line graph of a vase of $n$ flowers is an ear animal with $k=0$, $a=b=1$, and $c=n$. 
The line graph of an $(m,n)$-double star is an ear animal with $k=2$, $a=b=c=0$, 
$d_1=m$, and $d_2=n$.
The line graph of an isolated vertex is the empty graph $K_0$ and can thus be ignored.
Thus, the line graph $L(G)$ of $G$ is a graph each of whose components is an ear animal.
Therefore, by Theorem~\ref{thm:BBcharE}, $L(G)$~is $[B,B]$-perfect.
By the definition of line game-perfectness this means that
$G$ is line $[B,B]$-perfect.
%
%
%
%
\end{proof}

\subsection{Permitted for game $[B,-]$}
\label{subsec:B}

We start with some definitions. 

\begin{defin}[candy]
Let $m,n_1,n_2\in\IN$ and $m\ge1$. 
An \emph{$(m,n_1,n_2)$-candy} is a graph consisting of a $K_{2,m}$
and $n_1$ vertices of degree 1 that are adjacent to a
vertex $v_1$ of degree $m$ of the $K_{2,m}$ and $n_2$ vertices of degree 1
that are adjacent to the vertex $v_2$ at distance 2 from $v_1$,
i.e., the graph has the vertex set
\[\{v_1,v_2\}\cup\{w_i\mid 1\le i\le m\}\cup\{x_j\mid 1\le j\le n_1\}
\cup\{y_j\mid 1\le j\le n_2\}\]
and the edge set
\[\{v_1w_i,w_iv_2\mid 1\le i\le m\}\cup\{x_jv_1\mid 1\le j\le n_1\}
\cup\{v_2y_j\mid 1\le j\le n_2\}.\]
A \emph{candy} is an $(m,n_1,n_2)$-candy 
and an \emph{empty candy} is a $(1,n_1,n_2)$-candy for some $m,n_1,n_2\in\IN$.
\end{defin}



\begin{defin}[shooting star]
Let $m,n\in\IN$. 
An \emph{$(m,n)$-shooting star} is a graph formed by a central vertex
$v$ with $n$ adjacent vertices of degree 1 and a pending $P_3$ and
another adjacent vertex $w$ that is adjacent to $m$ vertices of degree~1,
i.e., the graph has the vertex set
\[\{v,w,a,b\}\cup\{x_i\mid 1\le i\le m\}\cup\{y_j\mid 1\le j\le n\}\]
and the edge set
\[\{wv,va,ab\}\cup\{wx_i\mid 1\le i\le m\}\cup\{vy_j\mid 1\le j\le n\}.\]
A \emph{shooting star} is an $(m,n)$-shooting star for some $m,n\in\IN$. 
\end{defin}

\begin{defin}[double vase]
Let $n\in\IN$. 
A \emph{double vase of $n$ flowers} is a graph formed by a central vertex
$v$ with $n$ adjacent vertices of degree 1 and two pending triangles,
i.e., the graph has the vertex set
\[\{v,x_1,x_2,y_1,y_2\}\cup\{w_j\mid 1\le j\le n\}\]
and the edge set
\[\{vx_1,x_1x_2,x_2v,vy_1,y_1y_2,y_2v\}\cup\{vw_j\mid 1\le j\le n\}.\]
A \emph{double vase} is a double vase of $n$ flowers for some $n\in\IN$;
if $n=0$, it is a 2-windmill.
\end{defin}

\begin{defin}[amaryllis]
Let $m,n\in\IN$. 
An \emph{$(m,n)$-amaryllis} is a graph formed by a central vertex
$v$ with $n$ adjacent vertices of degree 1 and a pending triangle and
another adjacent vertex $w$ that is adjacent to $m$ vertices of degree~1,
i.e., the graph has the vertex set
\[\{v,w,c_1,c_2\}\cup\{x_i\mid 1\le i\le m\}\cup\{y_j\mid 1\le j\le n\}\]
and the edge set
\[\{wv,vc_1,c_1c_2,c_2v\}\cup\{wx_i\mid 1\le i\le m\}\cup\{vy_j\mid 1\le j\le n\}.\]
An \emph{amaryllis} is an $(m,n)$-amaryllis for some $m,n\in\IN$;
if $m=0$, it is a vase of $n+1$ flowers.
\end{defin}

\begin{figure}[htbp]
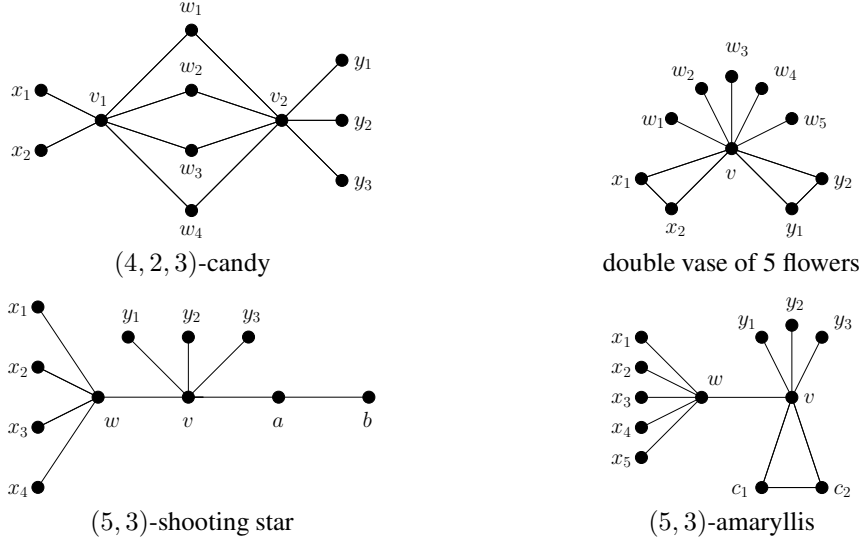

\begin{center}
\begin{minipage}[b]{0.48\textwidth}
\begin{center}
\includegraphics[scale=0.4]{Xcandyexample}

$(4,2,3)$-candy
\end{center}
\end{minipage}
\begin{minipage}[b]{0.48\textwidth}
\begin{center}
\includegraphics[scale=0.4]{Xdoublevase}

double vase of 5 flowers
\end{center}
\end{minipage}
\end{center}

\begin{center}
\begin{minipage}[b]{0.48\textwidth}
\begin{center}
\includegraphics[scale=0.4]{Xshootingstar}

$(5,3)$-shooting star
\end{center}
\end{minipage}
\begin{minipage}[b]{0.48\textwidth}
\begin{center}
\includegraphics[scale=0.4]{Xamaryllis}

$(5,3)$-amaryllis
\end{center}
\end{minipage}
\end{center}
\caption{\label{deffigB}A candy, a shooting star, a double vase, and an amaryllis.}
\end{figure}

We prove first that Alice wins the $[B,-]$-colouring game 
with $\omega(L(G))$ colours on the
configurations $G$ needed for Theorem~\ref{thm:lineB}.
In the proofs we refer to the notation given above.

A component of a graph is \emph{nontrivial} if it contains an edge.


\begin{lem}\label{lem:AwinsBcandy}\label{BniceCandy}
Graphs whose single nontrivial component is a candy
are line $[B,-]$-nice.
\end{lem}

\begin{proof}
Let $G$ be a graph that consists of a $(m,n_1,n_2)$-candy ($m\ge1$) and possibly some isolated
vertices. Then the line graph $L(G)$ of $G$ is an expanded cocobi with
$a=n_1$, $d=n_2$, $k=m$, $b_1,\ldots,b_k=1$, and $c_1,\ldots,c_k=1$.
Therefore, by Proposition~\ref{prop:BEfive}, the line graph $L(G)$ is
$[B,-]$-perfect. Thus $G$ is line $[B,-]$-perfect.
\end{proof}

We will use Lemma~\ref{BniceCandy} not only for the proof of
Theorem~\ref{thm:lineB}, but also for the proofs of
Theorem~\ref{thm:lineBA} and Theorem~\ref{thm:lineAA}.
Note that the proof given above relies on the results by \cite{lock} and \cite{andreslock}.

\begin{lem}\label{lem:AwinsBshootingstar}\label{BniceShootingStar}
Graphs whose single nontrivial component is a shooting star
are line $[B,-]$-nice.
\end{lem}

\begin{proof}
Let $G$ be a graph that consists of a $(m,n)$-shooting star 
and possibly some isolated
vertices. Then the line graph $L(G)$ of $G$ is an expanded bull with
$a=m$, $b=d=1$, and $c=n$.
Therefore, by Proposition~\ref{prop:BEsix}, the line graph $L(G)$ is
$[B,-]$-perfect. Thus $G$ is line $[B,-]$-perfect.
\end{proof}

\begin{lem}\label{lem:AwinsBdoublevase}\label{BniceDoubleVase}
Graphs whose single nontrivial component is a double vase
are line $[B,-]$-nice.
\end{lem}

\begin{proof}
Let $G$ be a graph that consists of a double vase of $n$ flowers
and possibly some isolated
vertices. Then the line graph $L(G)$ of $G$ is an expanded bull with
$a=1$, $b=d=2$, and $c=n$.
Therefore, by Proposition~\ref{prop:BEsix}, the line graph $L(G)$ is
$[B,-]$-perfect. Thus $G$ is line $[B,-]$-perfect.
\end{proof}

\begin{lem}\label{lem:AwinsBamaryllis}\label{BniceAmaryllis}
Graphs whose single nontrivial component is an amaryllis
are $[B,-]$-nice.
\end{lem}

\begin{proof}
Let $G$ be a graph that consists of an $(m,n)$-amaryllis 
and possibly some isolated
vertices. Then the line graph $L(G)$ of $G$ is an expanded bull with
$a=m$, $b=1$, $c=n$, and $d=2$.
Therefore, by Proposition~\ref{prop:BEsix}, the line graph $L(G)$ is
$[B,-]$-perfect. Thus $G$ is line $[B,-]$-perfect.
\end{proof}

\subsection{Permitted for game $[B,A]$}
\label{subsec:BA}

\begin{defin}[star book]\label{def:starbook}
Let $m,n_1,n_2\in\IN$ and $m\ge1$. 
An \emph{$(m,n_1,n_2)$-star book} is a graph consisting of a $K_{2,m}$
and $n_1$ vertices of degree 1 that are adjacent to a
vertex $v_1$ of degree $m$ of the $K_{2,m}$ and $n_2$ vertices of degree 1
that are adjacent to the vertex $v_2$ at distance 2 from $v_1$,
furthermore there is an additional edge $v_1v_2$,
i.e., the graph has the vertex set
\[\{v_1,v_2\}\cup\{w_i\mid 1\le i\le m\}\cup\{x_j\mid 1\le j\le n_1\}
\cup\{y_j\mid 1\le j\le n_2\}\]
and the edge set
\[\{v_1v_2\}\cup\{v_1w_i,w_iv_2\mid 1\le i\le m\}\cup\{x_jv_1\mid 1\le j\le n_1\}
\cup\{v_2y_j\mid 1\le j\le n_2\}.\]
A \emph{star book} is an $(m,n_1,n_2)$-star book 
for some $m,n_1,n_2\in\IN$.
Furthermore, a \emph{star book with $m$ book sheets} 
is an $(m,n_1,n_2)$-star book for some $n_1,n_2\in\IN$. 
\end{defin}

Thus, a star book is like a candy, but with an additional edge $v_1v_2$.

\begin{defin}[diamond of flowers]\label{def:diamondofflowers}
Let $n\in\IN$.
A \emph{diamond of $n$ flowers} is 
constructed from a diamond by attaching $n$ leaves
to a vertex $v$ of degree 2, i.e., the graph has the vertex set
\[\{v,u_1,u_2,w\}\cup\{x_j\mid 1\le j\le n\}\]
and the edge set
\[\{vu_1,vu_2,u_1u_2,wu_1,wu_2\}\cup\{vx_j\mid 1\le j\le n\}.\]
A \emph{diamond of flowers} is a diamond of $n$ flowers for some $n\in\IN$.
\end{defin}

\begin{defin}[tetrahedron of flowers]\label{def:tetrahedronofflowers}
Let $n\in\IN$.
A \emph{tetrahedron of $n$ flowers} is 
constructed from a $K_4$ by attaching $n$ leaves
to one of its vertices, i.e., the graph has the vertex set
\[\{v,u_1,u_2,u_3\}\cup\{x_j\mid 1\le j\le n\}\]
and the edge set
\[\{vu_1,vu_2,vu_3,u_1u_2,u_1u_3,u_2u_3\}\cup\{vx_j\mid 1\le j\le n\}.\]
A \emph{tetrahedron of flowers} is a tetrahedron of $n$ flowers for some $n\in\IN$.
\end{defin}

\begin{defin}[single galaxy]\label{def:singlegalaxy}
Let $k,\ell\in\IN$. A \emph{$(k,\ell)$-single galaxy} 
consists of a central vertex $v$ and $k$ pending triangles and $\ell$ pending $P_3$, 
i.e., the graph
has the vertex set
\[\{v\}\cup\{c_s,d_s\mid 1\le s\le k\}\cup\{x_{t},y_{t}\mid 1\le t\le \ell\}\]
and the edge set
\[\{vc_s,vd_s,c_sd_s\mid 1\le s\le k\}
\cup\{vx_t,x_ty_t\mid 1\le t\le \ell\}.\]
A \emph{single galaxy} is a $(k,\ell)$-single galaxy for some $k,\ell\in\IN$.
\end{defin}

Note that a $(0,0)$-single galaxy is an isolated vertex.

\begin{defin}[double galaxy]\label{def:doublegalaxy}
Let $k,\ell,m,n\in\IN$. A \emph{$(k,\ell,m,n)$-double galaxy} 
consists of a central vertex $v$ and $k$ pending triangles, $\ell$ pending~$P_3$,
a pending $(m+1)$-star, and $n$ pending~$P_2$, i.e., the graph
has the vertex set
\begin{eqnarray*}
&&\{v,z\}\cup\{c_s,d_s\mid 1\le s\le k\}\cup\{x_{t},y_{t}\mid 1\le t\le \ell\}\\
&\cup&\{u_i\mid 1\le i\le m\}\cup\{w_j\mid 1\le j\le n\}
\end{eqnarray*}
and the edge set
\begin{eqnarray*}
&&\{vz\}\cup\{vc_s,vd_s,c_sd_s\mid 1\le s\le k\}
\cup\{vx_t,x_ty_t\mid 1\le t\le \ell\}\\
&\cup&\{zu_i\mid 1\le i\le m\}\cup\{vw_j\mid 1\le j\le n\}
\end{eqnarray*}
A \emph{double galaxy} is a $(k,\ell,m,n)$-double galaxy for some $k,\ell,m,n\in\IN$.
\end{defin}

Note that a $(0,0,m,n)$-double galaxy is a double star.

\begin{figure}[htbp]
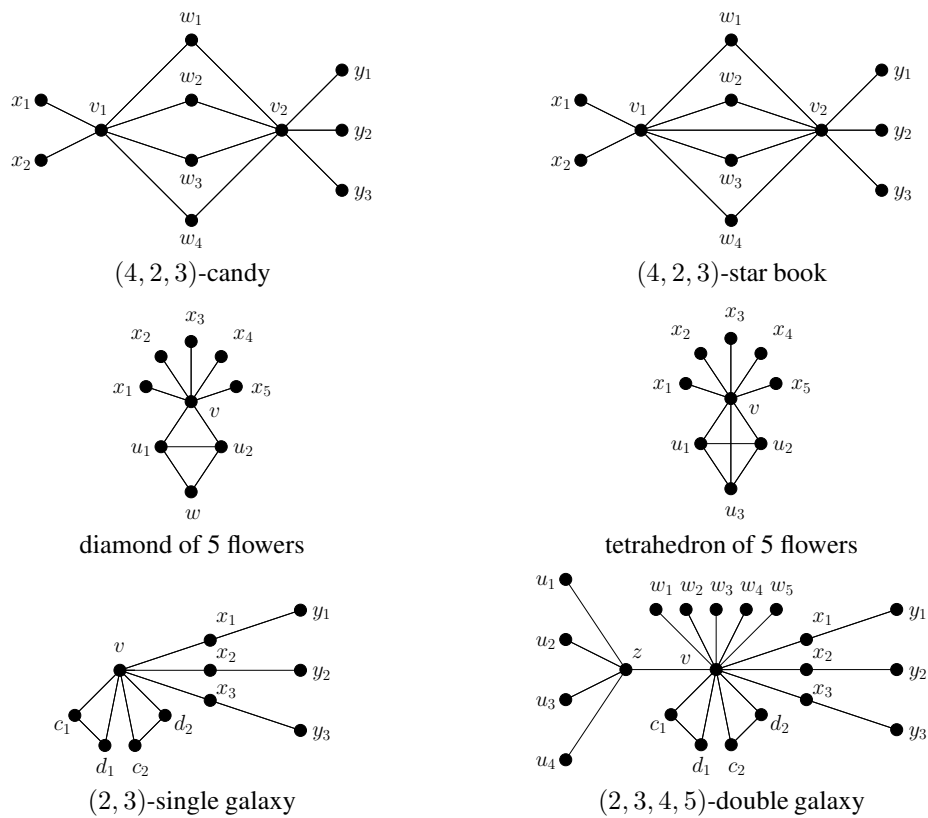

\begin{center}
\begin{minipage}[b]{0.48\textwidth}
\begin{center}
\includegraphics[scale=0.4]{Xcandyexample}

$(4,2,3)$-candy
\end{center}
\end{minipage}
\begin{minipage}[b]{0.48\textwidth}
\begin{center}
\includegraphics[scale=0.4]{Xstarbook}

$(4,2,3)$-star book
\end{center}
\end{minipage}
\end{center}

\begin{center}
\begin{minipage}[b]{0.48\textwidth}
\begin{center}
\includegraphics[scale=0.4]{Xdiamondofflowers}

diamond of 5 flowers
\end{center}
\end{minipage}
\begin{minipage}[b]{0.48\textwidth}
\begin{center}
\includegraphics[scale=0.4]{Xtetrahedronofflowers}

tetrahedron of 5 flowers
\end{center}
\end{minipage}
\end{center}

\begin{center}
\begin{minipage}[b]{0.48\textwidth}
\begin{center}
\includegraphics[scale=0.4]{Xsinglegalaxy}

$(2,3)$-single galaxy
\end{center}
\end{minipage}
\begin{minipage}[b]{0.48\textwidth}
\begin{center}
\includegraphics[scale=0.4]{Xdoublegalaxy}

$(2,3,4,5)$-double galaxy
\end{center}
\end{minipage}
\end{center}
\caption{\label{deffigBA}A candy, a star book, a diamond of flowers,
a tetrahedron of flowers, a single galaxy and a double galaxy.}
\end{figure}


\begin{lem}\label{lem:AwinsBAstarbook}\label{BAniceStarBook}
A star book is line $[B,A]$-nice.
\end{lem}

\begin{proof}
A triangle is line $[B,A]$-nice, trivially.
We describe a winning strategy for Alice in the $[B,A]$-edge colouring game
played on a star book $S$ that is not a triangle with
\[c=\max\{m+n_1+1,m+n_2+1\}\]
colours. We use the vertex labels from Definition~\ref{def:starbook}.

Whenever Bob does not colour the universal edge $v_1v_2$, then Alice
follows her winning strategy for the candy $S-v_1v_2$ with $c-1$ colours, 
which exists by Lemma~\ref{BniceCandy}.
If Bob colours $v_1v_2$, then Alice misses her turn.
After that she uses again her strategy for the candy $S-v_1v_2$.
This strategy is feasible since the colour of $v_1v_2$ cannot be used
on any other edge, since $v_1v_2$ is adjacent to every other edge.
\end{proof}

We remark that the same strategy shows that a star book is line $[A,-]$-nice.
Namely, Alice should colour $v_1v_2$ in her first move in the edge colouring
game $[A,-]$.
However, it is not \rd{line} $[A,-]$-perfect (except in the trivial case
when it is a vase of flowers) since it contains a $P_5$ or a $C_4$,
which are forbidden configurations for the edge colouring game $[A,-]$.


\begin{lem}\label{lem:AwinsBAdiamond}\label{BAniceDiamondOfFlowers}
A diamond of flowers is line $[B,A]$-nice.
\end{lem}

\begin{proof}
Let $D_n$ be a diamond of $n$ flowers with the vertex labels from
Definition~\ref{def:diamondofflowers}.
We describe a winning strategy for Alice with $c:=\max\{3,n+2\}$ colours
for the $[B,A]$-edge colouring game played on $D_n$. Consider Bob's first move.
\begin{itemize}
\item If Bob colours $u_1u_2$, then, if $n\ge1$, Alice colours $e=vx_1$
with the same colour, and if $n=0$, she misses her turn. And vice-versa,
if Bob colours a star edge $e=vx_j$, then Alice colours $u_1u_2$ with the
same colour. In all cases this colour may not be used for any other edge.
Therefore, after that, Alice may follow her winning strategy with $c-1$ 
colours for
the $[B,-]$-edge colouring game played on the candy $D_n-\{u_1u_2,e\}$,
which exists by Lemma~\ref{BniceCandy}.
\item If Bob colours $vu_i$ for some $i\in\{1,2\}$, then Alice colours
$wu_{3-i}$ with the same colour, and vice-versa. In any case, this colour
may not be used for any other edge. Thus, after that,
Alice may follow her winning strategy for the $[B,-]$-edge colouring
game with $c-1$ colours played on the shooting star $D_n-\{vu_i,wu_{3-i}\}$,
which exists by Lemma~\ref{BniceShootingStar}. 
\end{itemize}
Thus, in any case, Alice wins.
%
\end{proof}


The next lemma is very similar to the preceeding one.

\begin{lem}\label{lem:AwinsBAtetra}\label{BAniceTetrahedronOfFlowers}
A tetrahedron of flowers is line $[B,A]$-nice.
\end{lem}

\begin{proof}
Let $T_n$ be a tetrahedron of $n$ flowers with the vertex labels from
Definition~\ref{def:tetrahedronofflowers}.
We describe a winning strategy for Alice with $c:=n+3$ colours
for the $[B,A]$-edge colouring game played on $T_n$. Consider 
again Bob's first move.
\begin{itemize}
\item If Bob colours $u_1u_2$, then Alice colours $vu_3$ with the same colour
and vice-versa.
In all cases this colour may not be used for any other edge.
Therefore, after that, Alice may follow her winning strategy with \mbox{$c-1$} 
colours for
the $[B,-]$-edge colouring game played on the candy \mbox{$T_n-\{u_1u_2,vu_3\}$},
which exists by Lemma~\ref{BniceCandy}.
\item If Bob colours $vu_i$ for some $i\in\{1,2\}$, then Alice colours
$u_3u_{3-i}$ with the same colour, and vice-versa. In any case, this colour
may not be used for any other edge. Thus, after that,
Alice may follow her winning strategy for the $[B,-]$-edge colouring
game with \mbox{$c-1$} colours played on the candy $T_n-\{vu_i,u_3u_{3-i}\}$,
which exists by Lemma~\ref{BniceCandy}.
\item If Bob colours a star edge $vx_i$ for some $i\in\IN$ with $1\le i\le n$,
then Alice colours $u_1u_2$ with the same colour. This colour may not
be used for any other edge. Thus, after that,
Alice may follow her winning strategy for the $[B,A]$-edge colouring game
with $c-1$ colours played on the star book $T_n-\{vx_i,u_1u_2\}$,
which exists by Lemma~\ref{BAniceStarBook}. 
\end{itemize}
Thus, in any case, Alice wins.
%
%
\end{proof}



\rd{A $P_3$ is \emph{pending} at a vertex $v$ of a graph $G$
if the $P_3$ is an induced subgraph of $G$, 
and $v$ is a vertex of degree 1 in the $P_3$,
and, if $v$ is deleted, the two other vertices of the $P_3$
are disconnected with the vertices of $G-v$ that are not in the~$P_3$.
Analoguously, a triangle is \emph{pending} at a vertex $v$ of a graph~$G$
if the triangle is an induced subgraph of~$G$, 
and $v$ is one of the three vertices of the triangle,
and, if $v$ is deleted, the two other vertices of the triangle
are disconnected with the vertices of $G-v$ that are not in the triangle.}

For a (single or double) galaxy with the notation of
Definition~\ref{def:singlegalaxy} resp.\ Definition~\ref{def:doublegalaxy},
we call a $P_3$ pending at $v$ resp.\ a triangle (pending at $v$) a
\emph{pending object}. The \rd{(one or two)} edges \rd{of the pending object}
incident with $v$ are called \emph{star edges},
the other edge of \rd{the} pending object is called \emph{matching edge}.

\begin{lem}\label{lem:AwinsBAsgal}\label{BAniceSingleGalaxy}
A single galaxy is line $[B,A]$-nice.
\end{lem}

\begin{proof}
Let $G$ be a single galaxy with the notation from Definition~\ref{def:singlegalaxy}.

If the number $k+\ell$ of pending objects of $G$ is at most 2, then
$G$ is an isolated vertex, a $P_3$, a triangle, a $P_5$, an amaryllis, or
a double vase, thus, in any case, line $[B,-]$-nice, which implies by
{Observation~\ref{obs:comparechigL}} 
that $G$ is line $[B,A]$-nice.

Otherwise, the maximum degree $\Delta(G)$ of $G$ is at least 3, and Alice has
the following winning strategy with $\Delta(G)$ colours in the
$[B,A]$-edge colouring game played on~$G$. She {may arbitrarily number} the pending objects
$O_1,O_2,\ldots,O_{k+\ell}$ and \rd{perform} the following pairing strategy.

\begin{itemize}
\item
If Bob colours the matching edge of the pending object $O_j$, then
Alice colours a star edge of the pending object $O_{j+1\mod{k+\ell}}$
with the same colour, if possible. If it is not possible, she uses
a new colour for such a star edge.
\item
If Bob colours the first star edge of the pending object $O_j$ (a triangle or a $P_3$), then
Alice colours the matching edge of the pending object $O_{j-1\mod{k+\ell}}$
with the same colour.

\item
If Bob colours the second star edge of the pending object (a triangle) $O_j$, 
then Alice misses her turn.

\end{itemize}

Note that by this strategy, colouring a star edge with the same colour
in the first case is only not possible if the colour has been already used
for a star edge and a matching edge. 
Note further that after Alice's moves,
a star edge is either adjacent to an uncoloured matching edge or to a matching edge coloured
in a colour of another star edge.
Furthermore, by the pairing strategy, whenever Bob colours a matching edge
with a new colour, then there is a nonadjacent uncoloured star edge left
that can be coloured with the same colour.
Thus every unsafe edge (the star edges)
can be coloured feasibly.

\end{proof}



\begin{lem}\label{lem:AwinsBAdgal}\label{BAniceDoubleGalaxy}
A double galaxy is line $[B,A]$-nice.
\end{lem}

\begin{proof}
Let $G$ be a double galaxy with the notation from Definition~\ref{def:doublegalaxy}.

If the number $k+\ell$ of pending objects of $G$ is at most 1, then
$G$ is a double star, a shooting star or an amaryllis, 
thus, in any case, line $[B,-]$-nice, which implies by
{Observation~\ref{obs:comparechigL}} 
that $G$ is line $[B,A]$-nice.


Otherwise, the degree of $v$ is at least 3, and Alice 
uses an extension of the strategy for a single galaxy.
Note that the maximum degree of $G$ is
\[\Delta(G)=\max\{m+1,2k+\ell+n+1\}.\]
We describe a winning strategy for Alice with $\Delta(G)$ colours
in the $[B,A]$-edge colouring game played on~$G$.


The only unsafe edges are the star edges
of pending objects and the edge~$vz$.
Alice {may arbitrarily number} the pending objects
$O_1,O_2,\ldots,O_{k+\ell}$ and performs basically
the same pairing strategy as in the proof of Lemma~\ref{BAniceSingleGalaxy}
with only small extensions, as described in the following.

\begin{itemize}
\item
If Bob colours the matching edge of the pending object $O_j$, then,
if this was the first such move and the edge $vz$ is still uncoloured,
Alice colours $vz$ with the same colour (if possible, or a new colour otherwise); 
otherwise,
Alice colours a star edge of the pending object $O_{j+1\mod{k+\ell}}$
with the same colour, if possible. If it is not possible, she uses
a new colour for such a star edge.

\item
If Bob colours the first star edge of the pending object $O_j$ 
and there is still a pending object with only uncoloured star edges,
then
Alice colours the matching edge of the pending object $O_{j-1\mod{k+\ell}}$
with the same colour. If the matching edge is already coloured, then Alice
misses her turn. 
\item
If Bob colours the first star edge of the pending object $O_j$ 
and there is no pending object with only uncoloured star edges left,
then Alice colours $vz$
with a new colour (if $vz$ is still uncoloured) or misses her turn (if $vz$
is already coloured).
\item
If Bob colours the edge $vz$, an edge $vx_j$ or the second star edge of the pending object (a triangle) $O_j$, 
then Alice misses her turn.
\item
If Bob colours an edge $zu_i$, then Alice colours $vz$ (if $vz$ is still uncoloured)
or misses her turn (otherwise).
\end{itemize}

This strategy has the same properties as the strategy for the single galaxy
in the proof of Lemma~\ref{BAniceSingleGalaxy}, and, in addition, it guarantees
that the edge $vz$ is coloured before it is in danger to be infeasible
for any colour.
%
%
%
%
%
\end{proof}

\subsection{Permitted for game $[A,A]$}
\label{subsec:AA}

\begin{defin}[full tree]\label{def:fulltree}
Let $n,m_1,m_2\in\IN$.
An \emph{$(n,m_1,m_2)$-full tree} is based on a path $P_3$, 
where there are $m_1$
(respectively, $n$, $m_2$) \rd{leaves} attached its three vertices,
i.e., the graph has the vertex set
\[\{w_1,v,w_2\}\cup\{x_i\mid 1\le i\le m_1\}\cup\{y_j\mid 1\le j\le n\}\cup
\{z_i\mid 1\le i\le m_2\}\]
and the edge set
\[\{w_1v,vw_2\}\cup\{w_1x_i\mid 1\le i\le m_1\}\cup\{vy_j\mid 1\le j\le n\}
\cup\{w_2z_i\mid 1\le i\le m_2\}.\]
A \emph{{full tree}} 
is an $(n,m_1,m_2)$-full tree for some 
{$n,m_1,m_2\in\IN$}.
\end{defin}

Note that an $(n,m_1,1)$-full tree is a shooting star;
{a $(0,m_1,m_2)$-full tree is an empty candy;}
an $(n,m_1,0)$-full tree is a double star;
and an $(n,0,0)$-full tree is a star. 
{These trivial configurations are excluded in the next definition.}

\begin{defin}[full tree of type $E_1$]\label{def:eeone}
{A \emph{full tree of type $E_1$} is a full tree that is neither a shooting star nor a candy
nor a double star.}
\end{defin}

{By Definition~\ref{def:eeone}, a full tree of type $E_1$ does not belong to the permitted configurations
for game $[B,A]$.}

\begin{defin}[satellite]\label{def:satellite}
Let $m_1,m_2\in\IN$.
An \emph{$(m_1,m_2)$-satellite} is based on a triangle $K_3$, 
where there are $m_1$
(respectively, $m_2$) \rd{leaves} attached to two of its three vertices
and exactly one leaf is attached to its third vertex,
i.e., the graph has the vertex set
\[\{w_0,w_1,w_2,y\}
\cup\{z_{1,i}\mid 1\le i\le m_1\}
\cup\{z_{2,i}\mid 1\le i\le m_2\}\]
and the edge set
\[\{w_0w_1,w_0w_2,w_1w_2,w_0y\}
\cup\{w_1z_{1,i}\mid 1\le i\le m_1\}
\cup\{w_2z_{2,i}\mid 1\le i\le m_2\}.\]
A \emph{{satellite}} 
is an $(m_1,m_2)$-satellite for some $m_1,m_2\in\IN$.
\end{defin}

{Note that an $(m_1,0)$-satellite is a special star book (with one book sheet).
Such a trivial configuration is excluded in the next definition.}

\begin{defin}[satellite of type $E_2$]\label{def:eetwo}
{A \emph{satellite of type $E_2$} is a satellite that is not a star book.}
\end{defin}

{By Definition~\ref{def:eetwo}, a satellite of type $E_2$ does not belong to the permitted configurations
for game $[B,A]$.}

\begin{figure}[htbp]
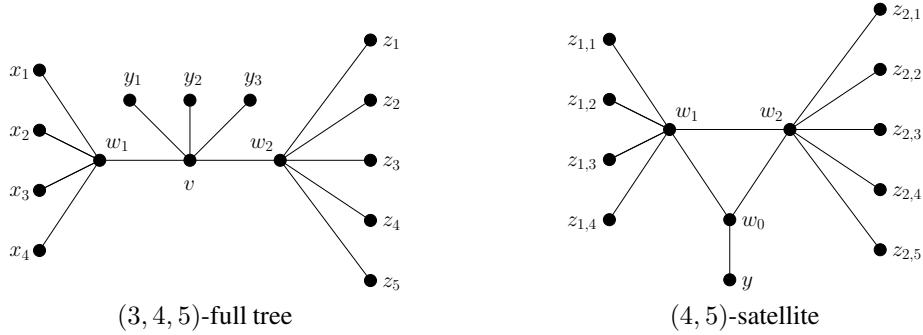

\begin{center}
\begin{minipage}[b]{0.48\textwidth}
\begin{center}
\includegraphics[scale=0.4]{Xfulltree}

$(3,4,5)$-full tree
\end{center}
\end{minipage}
\begin{minipage}[b]{0.48\textwidth}
\begin{center}
\includegraphics[scale=0.4]{Xsatellite}

$(4,5)$-satellite
\end{center}
\end{minipage}
\end{center}
\caption{\label{deffigAA}A full tree and a satellite}
\end{figure}


\begin{lem}\label{AAniceFullTree}
A {full tree} is line $[A,A]$-nice.
\end{lem}

\begin{proof}
Let $G$ be a full tree with the vertex labels from Definition~\ref{def:fulltree}.
Note that the only possibly unsafe edges are $vw_1$ and $vw_2$.
We describe a winning strategy for Alice with $\Delta(G)$ colours
in the $[A,A]$-edge colouring game played on~$G$.

If there is at most one unsafe edge, then Alice colours this edge (if any)
in her first move and wins trivially. In the following, we assume that there
are two unsafe edges.

If $G$ is the path $P_5$, which is a trivial shooting star, 
then Alice {misses her first turn and} has a trivial winning strategy 
{for the $[B,-]$-edge colouring game} by Lemma~\ref{BniceShootingStar}.

Otherwise, $\Delta(G)\ge3$, and so the number of colours in the game is
at least~3. Therefore, Alice may ensure that the 
%
two unsafe edges 
$vw_1$ and $vw_2$ are coloured after her first two moves, which is possible
with three colours since Bob has only one move \rd{in between}.

Thus, in any case, Alice wins. 
\end{proof}


\begin{lem}\label{AAniceSatellite}
A satellite is line $[A,A]$-nice.
\end{lem}

\begin{proof}
Let $G$ be a satellite with the vertex labels from Definition~\ref{def:satellite}.
Note that the only possibly unsafe edges are the triangle edges $w_0w_1$,
$w_0w_2$ and $w_1w_2$.
We describe a winning strategy for Alice with $c=\Delta(G)$ colours
in the $[A,A]$-edge colouring game played on~$G$.

Alice misses her first turn and then reacts as follows on Bob's first move.
\begin{itemize}
\item If Bob colours $w_1w_2$, then Alice colours $w_0y$ with the same
colour, and vice-versa. 
Now this colour may not be used any more.
Note that the graph $G-\{w_1w_2,w_0y\}$ is an empty candy and we have
\[\Delta(G-\{w_1w_2,w_0y\})=\Delta(G)-1=c-1.\]
Therefore, if Alice follows her winning strategy for the
$[B,-]$-edge colouring game with $c-1$ colours played on the empty candy
$G-\{w_1w_2,w_0y\}$, which exists by Lemma~\ref{BniceCandy}, then Alice will win.
\item Let $s\in\{1,2\}$. If Bob colours $w_0w_s$, then, if $m_{3-s}\ge1$, Alice colours 
the star edge $e:=w_{3-s}z_{3-s,1}$
with the same colour,
and if $m_{3-s}=0$, Alice misses her turn. And vice-versa, if Bob colours a
star edge $e:=w_{3-s}z_{3-s,i}$ for some $i\in\IN$ with $1\le i\le m_{3-s}$, then
Alice colours $w_0w_s$ with the same colour.
Now this colour may not be used any more.
Note that the graph 
\[G':=\left\{\begin{array}{ll}
G-\{w_0w_s,e\}&\text{if}\ m_{3-s}\ge1\\
G-\{w_0w_s\}&\text{if}\ m_{3-s}=0
\end{array}\right.\] 
is a shooting star and we have
\[\Delta(G')=\Delta(G)-1=c-1.\]
Therefore, if Alice follows her winning strategy for the
$[B,-]$-edge colouring game with $c-1$ colours played on the shooting star~$G'$, 
which exists by Lemma~\ref{BniceShootingStar}, then Alice will win.
\end{itemize}

Thus, in any case, Alice has a winning strategy.
\end{proof}

We remark that, since, in her strategy, Alice misses her first turn,
the proof of Lemma~\ref{AAniceSatellite}, indeed, shows
that a satellite is line $[B,A]$-nice. However, in general, a satellite is not
line $[B,A]$-perfect, since it may contain a 3-caterpillar $F_1$ as an edge-induced subgraph
(which can be seen by deleting the edge $w_1w_2$).


















\section{Proof of the structural characterisations}
\label{sec:proofs}

\subsection{Proof of Theorem~\ref{thm:lineAA}}

%

\rd{A basic concept in the proof of Theorem~\ref{thm:lineAA} is the concept
of a block of a graph.
Recall that a graph $G$ is \emph{2-connected} if it has at least 3 vertices and,
for any vertex $v$, if $v$ is deleted, the remaining graph $G-v$ is still connected.
A vertex~$v$ of a graph~$G$ is an \emph{articulation point} of~$G$
if deleting $v$ increases the number of components, i.e., $G-v$ has strictly more
components than~$G$.
A \emph{block} of a graph~$G$ is a maximal subgraph of~$G$ that 
is connected and contains
no articulation points.
Thus, a block is either a maximal 2-connected subgraph of~$G$ or a~$K_2$.
The edges of different blocks do not overlap, but several blocks may
share the same articulation point.}

\begin{proof}[of Theorem~\ref{thm:lineAA}]
We prove the equivalence by a ring closure. 
\begin{description}
\item[(1)$\Longrightarrow$(2)] 
We have to prove that $P_6,C_5, F_2,F_3$ {and $F_1\cup F_1$} are not line $[A,A]$-perfect.
It is sufficient to prove that they are not line $[A,A]$-nice.
This was proved in Lemma~\ref{notAAnicePsix} for path $P_6$, 
in Lemma~\ref{notAAniceFtwo} for the mini lobster $F_2$, and
in Lemma~\ref{notAAniceFthree} for the trigraph $F_3$.
The $C_5$ is not line perfect, thus it is not line $[A,A]$-perfect
(see Lemma~\ref{notAAniceCfive}). 
$F_1$ is not line $[B,A]$-nice by Lemma~\ref{notBAniceFone}, 
so $F_1\cup F_1$ is not line $[A,A]$-nice
(see Lemma~\ref{notAAnicetwoFone}).
%
%
%
%
%

\item[(2)$\Longrightarrow$(3)] 
Let $G$ be a graph that contains no $P_6$, $C_5, F_2,F_3,F_1\cup F_1$ as 
edge-induced subgraphs and let $H$ be a component of $G$.

Since $H$ does not contain $C_5$ \rd{or $P_6$}, the component $H$ is line perfect by
Theorem~\ref{thm:trotterchar}. 
Thus, by Theorem~\ref{thm:maffray}, every block of $H$ is either bipartite or
a $K_4$ or a triangular book $K_{1,1,m}$ for some $m\ge1$. 
Since $H$ contains no $P_6$ and no $C_5$,
the only possible cycles are triangles and 4-cycles.

We first observe that, 
among the blocks that $H$ contains,
there is at most 
one block which is
one of the following configurations:
$K_{1,1,m}$ with $m\ge2$ (a \emph{nontrivial} triangular book), 
\rd{a bipartite} 
block containing at least one 4-cycle, 
or~$K_4$.
Assume to the contrary \rd{that} $H$ contains two such (not necessarily different) configurations. 
Then 
{there are}
3 edges from each
of the two configurations {which}, together with the edges on a shortest path connecting
the two configurations, form {an edge-induced} $P_k$ with $k\ge7$, which contains a $P_6$,
contradicting~(2).

We observe further that if $H$ contains a block that is a triangle (i.e.,\ a \emph{trivial}
triangular book $K_{1,1,1}$), then $H$ cannot contain any of the configurations
$K_{1,1,m}$ with $m\ge2$, $K_4$, or a bipartite block containing at least one 4-cycle.
Assume to the contrary \rd{that} $H$ contains one of them. Then 
{there are}
3 edges from such
a configuration and 2 edges from the triangle {which}, together with the edges 
on a shortest path connecting
the configuration with the triangle, form {an edge-induced} $P_k$ with $k\ge6$, which contains a $P_6$,
contradicting~(2).

%


\begin{enumerate}[{Case }1:] 
\item $H$ contains a $K_4$.

Then there is a vertex $v$ \rd{of the $K_4$}, so that every edge of $H$ that is not
part of the $K_4$ is adjacent to $v$, since, otherwise, if outside of the
$K_4$ there is an
edge that is not adjacent to $v$ and an edge that is adjacent to $v$,
by using these two edges and 3 edges of the $K_4$ there is a
path or a cycle of lenght at least 5, thus $H$ contains 
{an edge-induced $P_6$ or $C_5$}, which
contradicts~(2). Thus $H$ is a {\bf tetrahedron of flowers}.


\item $H$ contains a nontrivial triangular book $K_{1,1,m}$ with $m\ge2$.

Let $v_1,v_2$ be the two vertices of degree $m+1$ and $w_1,\ldots,w_m$ the
vertices of degree 2 of the $K_{1,1,m}$. Let $e$ be an edge that does not
belong to the $K_{1,1,m}$.

\begin{figure}[htbp]
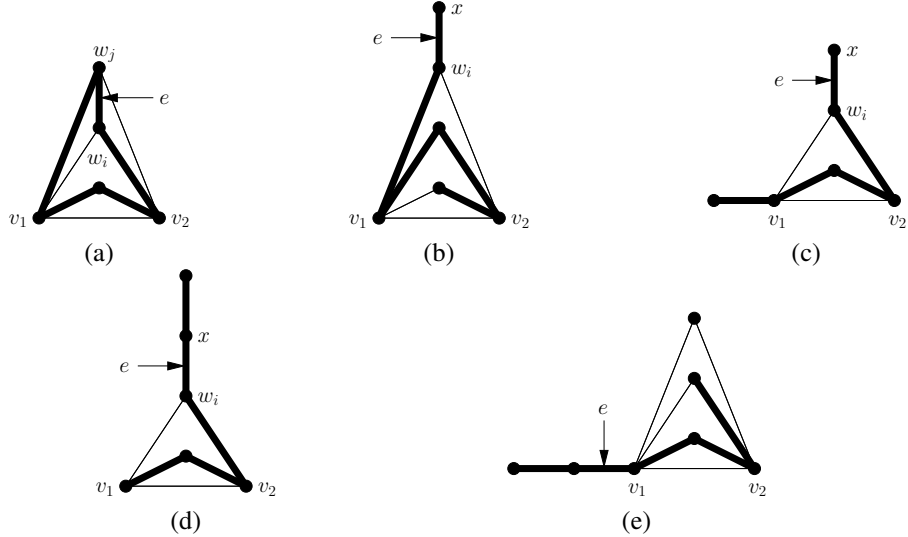

\begin{center}
\begin{minipage}[b]{0.3\textwidth}
\begin{center}
\includegraphics[scale=0.4]{proofBAaNEW}

(a)
\end{center}
\end{minipage}
\begin{minipage}[b]{0.3\textwidth}
\begin{center}
\includegraphics[scale=0.4]{proofBAbNEW}

(b)
\end{center}
\end{minipage}
\begin{minipage}[b]{0.35\textwidth}
\begin{center}
\includegraphics[scale=0.4]{proofBAcNEW}

(c)
\end{center}
\end{minipage}

\begin{minipage}[b]{0.3\textwidth}
\begin{center}
\includegraphics[scale=0.4]{proofBAdNEW}

(d)
\end{center}
\end{minipage}
\begin{minipage}[b]{0.5\textwidth}
\begin{center}
\includegraphics[scale=0.4]{proofBAeNEW}

(e)
\end{center}
\end{minipage}
\end{center}
\caption{\label{fig:proofAAcopy}Creating forbidden configurations in Case~2.}
\end{figure}

If $e$ connects two vertices $w_i$ and $w_j$ with
$i\neq j$, then, in case $m=2$, we have a $K_4$ and are thus in Case 1, which
was already discussed above, or, in case $m\ge3$, the graph contains 
{an edge-induced} $C_5$ (see
Figure~\ref{fig:proofAAcopy}~(a)), which is forbidden by~(2).

Now consider the case that $e$ is incident with some vertex $w_i$ and another
vertex $x$ that is not part of the $K_{1,1,m}$. If $m\ge3$, then $e$ together
with 4 edges of the $K_{1,1,m}$ form a $P_6$
(see Figure~\ref{fig:proofAAcopy}~(b)), which is forbidden by~(2).
If $m=2$, then there is neither an additional edge incident with $v_1$ {or} $v_2$
other than the edges of the $K_{1,1,m}$ nor an edge different from~$e$ incident
with $x$ since, otherwise, $H$ contains 
{an edge-induced} $P_6$
(see Figures~\ref{fig:proofAAcopy}~(c) and~(d)), which is forbidden by~(2).
Thus, in this case, $H$ is a {\bf diamond of flowers}.

If $e$ is incident with $v_1$ or $v_2$, then there is no edge adjacent to $e$
that is not part of the $K_{1,1,m}$ since, otherwise, $H$ contains a $P_6$
(see Figure~\ref{fig:proofAAcopy}~(e)), which is forbidden by~(2).
Thus, in this case, $H$ is a {\bf star book}.



\item $H$ is bipartite and contains a block with a $C_4$.

First note that the union of the 4-cycles in the block must form a $K_{2,m}$
with $m\ge2$. This is because it is the only possibility to combine
several 4-cycles in a block without creating cycles $C_{2k+6}$ with $k\in\IN$,
which are forbidden since its subgraph, the $P_6$, is forbidden by~(2).

Note that in the contradictions of Case~2 in Figure~\ref{fig:proofAAcopy} the
edge between $v_1$ and $v_2$ was \rd{not used}, which means they are still valid even 
if $v_1v_2$ is absent, which is the case $K_{2,m}$ here. Only for $m=2$ where there appeared a $K_4$, there would appear here a diamond (which is not bipartite). 
Thus, by the same
proof as in Case 2, 
$H$ is a {\bf candy}.




\item $H$ contains a block that is a triangle.

Consider one fixed triangle with vertices $v_1,v_2,v_3$.
Note that no $P_4$ may be pending at $v_i$ since,
otherwise, the three edges of the $P_4$ and two {of the} edges of {the} triangle would
form {an edge-induced} $P_6$, which is forbidden by~(2). 
Similarly, no $P_3$ resp.\ no other triangle
may be pending at $v_i$ when $P_2$ is pending at $v_j$ with $j\neq i$.

In the following, we disinguish \rd{between} the cases \rd{that} the number $a$
of vertices $v_1,v_2,v_3$ that have at least one pending $P_2$ is
2, 3 or at most 1.


\begin{enumerate}[{Subcase 4.}1:]
\item 
$a=2$.

\rd{The case that the block containing $v_1,v_2,v_3$ is a $K_4$ or a diamond
was already discussed in Case~1 or Case~2, respectively.
Thus we are left with the case that the block containing $v_1,v_2,v_3$
is a triangle.
Then
the component $H$ is a \textbf{star book} with exactly one triangle 
(and two of $v_1,v_2,v_3$ have at least one pending $P_2$).}

\item 
$a=3$.

Suppose now at least one $P_2$ is pending at every $v_i$. 
Since the trigraph $F_3$, in which every $v_i$ has two pending $P_2$s, 
is forbidden, at least one of $v_1,v_2,v_3$ 
has at most one pending $P_2$, which means that 
$H$ is a \textbf{satellite}. 

\item 
$a\le1$.

Suppose exactly one of $v_1,v_2,v_3$ has at least one pending $P_3$. 
Note that $H$ may contain more than one block which is a triangle. 
As discussed above, all these triangle blocks share exactly one vertex 
and we denote this by $v$. At most one pending star with at least 2 star edges 
may be pending at $v$ since,
otherwise, if there are two of them, two of the star edges of each pending star and
the edges of the pending stars that are incident with $v$ together with
{two of the edges} of a triangle would form 
{an edge-induced} mini lobster $F_2$, 
which is forbidden by~(2).
Thus, $H$ is a {\bf double galaxy} or a {\bf single galaxy}. 
\end{enumerate}

%


\item $H$ is a tree.

As the $P_6$ is forbidden, $H$ has diameter at most 4. If $H$ has diameter at most 3,
then $H$ is an isolated vertex (which is a {\bf single galaxy}) or
a double star (which is a {\bf double galaxy}). If $H$
has diameter 4, $H$ can be depicted by configuration $E$
in Figure~\ref{fig:E}. 
\begin{figure}[htbp]
\begin{center}
\includegraphics[scale=0.4]{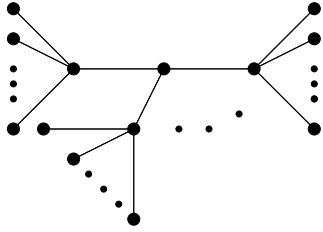}
\end{center}
\caption{\label{fig:E}A generic tree $E$ of diameter 4}
\end{figure}
Since the {mini lobster} $F_2$ is $[A,A]$-forbidden we conclude 
that, if $m\geq 3$ in configuration $E$, at most two of the stars 
{pending at the central vertex}
have more than one leaf. 
If there are exactly two such stars,
{then the other stars are simply pending $P_2$s, again 
since the mini lobster $F_2$ is forbidden, thus} 
$H$ is a \textbf{full tree}; otherwise, i.e., 
when there is at most one such star, 
$H$ is a \textbf{double galaxy} 
(or a \textbf{single galaxy})
with no triangles.
\end{enumerate}

%

Thus, in all cases we get a permitted configuration.
{Furthermore, since $F_1\cup F_1$ is forbidden, at most one of the
components may be a full tree of type~$E_1$ or a satellite of type~$E_2$ (as
both of the latter configurations contain a 3-caterpillar $F_1$ by definiton).
Therefore},
(3) holds.

\item[(3)$\Longrightarrow$(1)] 

The permitted configurations \rd{for each component} are line $[A,A]$-nice:
we have proved this 
for the candy in Lemma~\ref{BniceCandy},
for the star book in Lemma~\ref{BAniceStarBook},
for the single galaxy in Lemma~\ref{BAniceSingleGalaxy},
for the double galaxy in Lemma~\ref{BAniceDoubleGalaxy},
for the diamond of flowers in Lemma~\ref{BAniceDiamondOfFlowers},
for the tetrahedron of flowers in Lemma~\ref{BAniceTetrahedronOfFlowers},
for the full tree $E_1$ in Lemma~\ref{AAniceFullTree}, and
for the satellite in Lemma~\ref{AAniceSatellite}.
{With the exception of a full tree of type~$E_1$ and a satellite of type~$E_2$,
the permitted configurations are even line $[B,A]$-nice.}

{Let $G$ be a graph whose components are of the permitted types and where at
most one component is a \emph{special component}, namely a full tree of type~$E_1$
or a satellite of type~$E_2$.
Then, in her first move, Alice plays according to her strategy for the special component
(if there is a special component) or misses her turn (if there is no special component).
After that she reacts always in the component where Bob has played in his
previous move according to her strategy for the $[B,A]$-edge colouring game
(respectively, $[A,A]$-edge colouring game for the special component)
or she misses her turn (if the component where Bob has played in his previous move
is completely coloured).
By the \rd{lemmas} mentioned above, Alice will win. Thus $G$ is $[A,A]$-nice.}


In 
\rd{Lemma~\ref{lem:addhered} we will show}
that the permitted types
are hereditary. Thus the permitted types are line $[A,A]$-perfect.
{From this we conclude that $G$ is line $[A,A]$-perfect},
which proves (1).
\end{description}~\end{proof}

\begin{lem}\label{lem:addhered}
\rd{The permitted types for game $[A,A]$ are hereditary.}
\end{lem}

\begin{proof}
\rd{See Table~\ref{tab:AAhereditary} and its caption.}
\end{proof}

\begin{table}[htbp]
\begin{center}
\begin{tabular}{|l|l|}
\hline
Structure&Structure after deleting an edge\\
\hline
\hline
candy&candy\\
&two stars (which are galaxies)\\
\hline
star book&star book\\
&candy\\
&double star (which is a double galaxy)\\
\hline
diamond of flowers&diamond of flowers\\
&double galaxy\\
&star book\\
&candy\\
\hline
tetrahedron of flowers&tetrahedron of flowers\\
&diamond of flowers\\
&star book\\
\hline
single galaxy&double galaxy\\
&single galaxy \& $P_2$ (which is a double galaxy)\\
&single galaxy\\
\hline
double galaxy&double galaxy\\
&double galaxy \& star (which is a galaxy)\\
&single galaxy \& star (which is a galaxy)\\
\hline
{\bf satellite}&star book\\
&double galaxy\\
&{\bf full tree $E_1$}\\
&{\bf satellite}\\
\hline
{\bf full tree $E_1$}&{\bf full tree}\\
&double star \& star (which are galaxies)\\
\hline
\end{tabular}
\end{center}
\caption{\label{tab:AAhereditary}The permitted types for the edge-game $[A,A]$ are
hereditary. Note that, in particular, by deleting an edge, there will remain
at \rd{most} one satellite or full tree. \rd{Furthermore, in some cases
additionally there will be isolated vertices, which we do not mention
explicitly.}}
\end{table}

\rd{In order to illustrate the proof we depict four
characteristic situations for the candy when one edge is deleted
in Figure~\ref{fig:candyhered}.}

\begin{figure}[htbp]
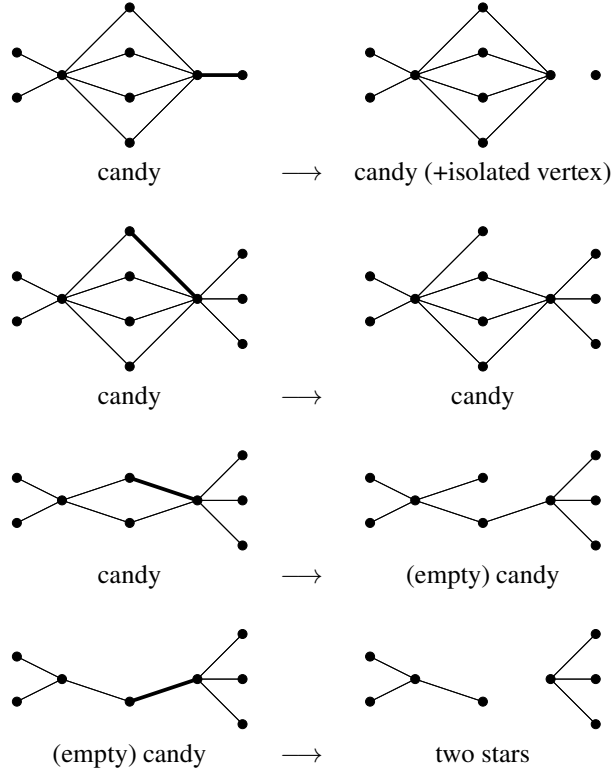

\begin{center}
\begin{tabular}{ccc}
\includegraphics[scale=0.3]{Xcandy1}&&\includegraphics[scale=0.3]{Xcandy2}\\
candy&$\longrightarrow$&candy (+isolated vertex)\\[3ex]
\includegraphics[scale=0.3]{Xcandy3}&&\includegraphics[scale=0.3]{Xcandy4}\\
candy&$\longrightarrow$&candy\\[3ex]
\includegraphics[scale=0.3]{Xcandy5}&&\includegraphics[scale=0.3]{Xcandy6}\\
candy&$\longrightarrow$&(empty) candy\\[3ex]
\includegraphics[scale=0.3]{Xcandy7}&&\includegraphics[scale=0.3]{Xcandy8}\\
(empty) candy&$\longrightarrow$&two stars
\end{tabular}
\end{center}
\caption{\label{fig:candyhered}\rd{Examples for deleting an edge in a candy}}
\end{figure}


\subsection{Proof of Theorem~\ref{thm:lineBA}}

\begin{proof}[of Theorem~\ref{thm:lineBA}]
We prove the equivalence by a ring closure. 
\begin{description}
\item[(1)$\Longrightarrow$(2)] 
We have to prove that a $P_6$, a $C_5$ and a 3-caterpillar $F_1$ are not 
line $[B,A]$-perfect.
Since, by Lemma~\ref{notAAnicePsix}
and Lemma~\ref{notAAniceCfive}, $P_6$ and $C_5$ are not line $[A,A]$-perfect,
by Observation~\ref{obs:compareclasses} they are not line $[B,A]$-perfect.
Thus, it is sufficient to prove that $F_1$ is not line $[B,A]$-nice.
This was proved in 
Lemma~\ref{notBAniceFone}.

\item[(2)$\Longrightarrow$(3)] 
Let $G$ be a graph that contains no $P_6$, $C_5$, {or} 3-caterpillar $F_1$ as 
{an edge-induced subgraph}.
Thus, in particular, the graph $G$ contains no $P_6$, $C_5$, mini lobster
$F_2$, trigraph $F_3$, and no $F_1\cup F_1$.

By Theorem~\ref{thm:lineAA}, every component of $G$ is a candy,
a star book, a diamond of flowers, a tetrahedron of flowers, a single galaxy,
a double galaxy, a full tree,
or a satellite.
Since a full tree {of type~$E_1$} and a satellite 
{of type~$E_2$ contain a 3-caterpillar} 
no component of $G$ is a full tree {of type~$E_1$} or a satellite 
{of type~$E_2$}. Thus (3) holds.

\item[(3)$\Longrightarrow$(1)] 

The permitted configurations are line $[B,A]$-nice:
we proved this 
for the candy in Lemma~\ref{BniceCandy},
for the star book in Lemma~\ref{BAniceStarBook},
for the single galaxy in Lemma~\ref{BAniceSingleGalaxy},
for the double galaxy in Lemma~\ref{BAniceDoubleGalaxy},
for the diamond of flowers in Lemma~\ref{BAniceDiamondOfFlowers}, and
for the tetrahedron of flowers in Lemma~\ref{BAniceTetrahedronOfFlowers}.

{Let $G$ be a graph whose components are of one of the permitted types
for game $[B,A]$.
Then Alice always reacts in the component where Bob has played
according to her strategy for the $[B,A]$-edge colouring game
(or she misses her turn if this component is completely coloured).
By the mentioned lemmata, Alice will win. 
Thus $G$ is line $[B,A]$-nice.}

Furthermore, the permitted configurations are hereditary, which can be
seen from the first six entries in Table~\ref{tab:AAhereditary}.
{From this we conclude that $G$ is line $[B,A]$-perfect,
which proves~(1).}
\end{description}
\end{proof}

\subsection{Proof of Theorem~\ref{thm:lineB}}

\begin{proof}[of Theorem~\ref{thm:lineB}]
We prove the equivalence by a ring closure. 
\begin{description}
%
%
\item[(1)$\Longrightarrow$(2)] 
This implication is part of Proposition~\ref{prop:lineB}.

\item[(2)$\Longrightarrow$(3)]
Let $G$ be a graph that fulfils (2), i.e., 
it contains no $P_5\cup P_2$, $C_4\cup P_2$, $P_6$, $C_5$,
bull, diamond, {or} 3-caterpillar as {an edge-induced subgraph}.

By (2), the graph $G$, in particular, contains no $P_6$, $C_5$, 3-caterpillar.
Thus, by Theorem~\ref{thm:lineBA}, each component of $G$ is a diamond of
flowers, a tetrahedron of flowers, a candy, a star book, a single galaxy,
or a double galaxy.
Let $H$ be a component of~$G$.

The component $H$ may neither be a diamond of flowers nor a tetrahedron
of flowers, since those two configurations contain a diamond as a subgraph,
which is forbidden by~(2).

Consider the case that $H$ is a star book. It may not contain more than one
book sheet, since otherwise it would contain a diamond, which is forbidden
by~(2). If $H$ has exactly one book sheet, it may not have star edges on
both sides, since otherwise it would contain a bull, which is forbidden by~(2).
Thus, in this case, the component $H$ is a {\bf vase of flowers}. 
If $H$ has no book sheet, then $H$ is a {\bf double star}.

Now consider the case that $H$ is a single galaxy or a double galaxy.
In both cases, the component $H$ has a vertex $v$ with $k_0$ pending $P_2$s,
$k_1$ pending $P_3$s, $k_2$ pending triangles, and $k_3$ pending stars,
where $k_0,k_1,k_2\ge0$ and $k_3\in\{0,1\}$. First note that
\begin{equation}\label{maineqgalaxy}
k_1+k_2+k_3\le2,
\end{equation}
since otherwise, if $k_1+k_2+k_3\ge3$, two of the pending objects would
contain a $P_5$ and the third would contain a $P_2$ that is not adjacent
with the $P_5$, thus $H$ would contain 
{an edge-induced} $P_5\cup P_2$, which is
forbidden by (2). So we may assume that (\ref{maineqgalaxy}) holds. 
We distinguish some cases.
\begin{itemize}
\item
If $k_1=2$ and $k_2=k_3=0$, then $H$ is a {\bf shooting star}.
\item
If $k_1=1$ and $k_2=k_3=0$, then $H$ is a {\bf double star}.
\item
If $k_1=k_2=1$ and $k_3=0$, then $H$ is an {\bf amaryllis}.
\item
If $k_1=k_3=1$ and $k_2=0$, then $H$ is a {\bf shooting star}.
\item
If $k_2\le2$ and $k_1=k_3=0$, then $H$ is a {\bf double vase},
a {\bf vase of flowers}, or a star (which is either a {\bf double star}
or an {\bf isolated vertex}).
\item
If $k_2=k_3=1$ and $k_1=0$, then $H$ is an {\bf amaryllis}.
\item
If $k_3=1$ and $k_1=k_2=0$, then $H$ is a {\bf double star}.
\end{itemize}

Finally, we have to prove that if one of the components of $G$ is a
candy, a shooting star, a double vase, or an amaryllis, but neither
a double star nor a vase of flowers, then $G$ has only one nontrivial
component. We observe that
\begin{itemize}
\item
a candy that is not a double star
contains a $P_5$ or a $C_4$, 
\item
a shooting star that is not a double star contains
a $P_5$,
\item
a double vase contains a $P_5$, and
\item
an amaryllis that is not a vase of flowers contains a $P_5$.
\end{itemize}
Thus, if there is such a component in the graph, then there may not
be another component that contains an edge, since otherwise $G$
would contain a $P_5\cup P_2$ or a $C_4\cup P_2$, which are forbidden
by~(2). Therefore, (3) holds.

\item[(3)$\Longrightarrow$(1)] 
Let $G$ be a graph fulfilling~(3).
Then, by Lemmata~\ref{lem:AwinsBBB},
\ref{BniceCandy},
\ref{BniceShootingStar},
\ref{BniceDoubleVase}, and
\ref{BniceAmaryllis},
the graph $G$ is line $[B,-]$-perfect, i.e., (1)~holds.
%
%
%
\end{description}
\end{proof}

\subsection{Proof of Theorem~\ref{thm:lineBB}}

\begin{proof}[of Theorem~\ref{thm:lineBB}]
We prove the equivalence by a ring closure. 
\begin{description}
\item[(1)$\Longrightarrow$(2)]
This follows from {${\cal LP}[B,B]\subseteq{\cal LP}[A,B]$ (Observation~\ref{obs:compareclasses})}.
 
\item[(2)$\Longrightarrow$(3)] 
This follows from {${\cal LP}[A,B]\subseteq{\cal LP}[A,-]$ (Observation~\ref{obs:compareclasses})}.

\item[(3)$\Longrightarrow$(4)] 
This implication is part of Proposition~\ref{prop:lineA}.

\item[(4)$\Longrightarrow$(5)] 
Let $G$ be a graph that contains neither $P_5$ nor $C_4$ as an edge-induced subgraph.
Thus, in particular $G$ contains no $P_5\cup P_2$, $C_4\cup P_2$, $P_6$,
$C_5$, bull, diamond, 3-caterpillar (since these configurations either
contain a $P_5$ or, in the case of $C_4\cup P_2$, a $C_4$). Thus
$G$ is line $[B,-]$-perfect. Therefore, by Theorem~\ref{thm:lineB}
every component of $G$ is a double star, a vase of flowers, an isolated vertex,
a candy, a shooting star, a double vase, or an amaryllis. Let $H$ be a
component of $G$.

If $H$ is a candy, then it must be an empty candy, since otherwise it
would contain a $C_4$, which is forbidden by~(4). Furthermore, it may
not have star edges at both sides, since otherwise it would contain a $P_5$,
which is forbidden by~(4). Thus $H$ is a {\bf double star}.

If $H$ is a shooting star, then it must have diameter 3, since otherwise
it would contain a $P_5$, which is forbidden by~(4). Thus $H$ is a
{\bf double star}.

Note that $H$ may not be a double vase as the two triangles contain 
{an edge-induced} $P_5$,
which is forbidden by~(4).

Furthermore, if $H$ is an amaryllis, then the pending star must be empty,
since otherwise two edges of the pending star and two edges of the triangle
would form a $P_5$, which is forbidden by~(4).
{Thus $H$ is a {\bf vase of flowers}.}

{We conclude that} (5) holds.

\item[(5)$\Longrightarrow$(1)]
We have to prove that graphs each component of which is a
double star, vase of flowers or isolated vertex is line $[B,B]$-perfect.
%
This was shown in Lemma~\ref{lem:AwinsBBB},
which proves the last implication of the theorem.
\end{description}
\end{proof}

\section{Final remarks}
\label{sec:final}

In this paper, we completely characterize line game-perfect graphs for all six
possible games. 

\subsection{\rd{Similar characterisations for vertex colouring games}}

Similar characterisations for game-perfect graphs
(where vertex games are considered instead of edge games) are only known
for the games $[B,B]$, $[A,B]$, $[A,-]$ and $[B,-]$. Thus the following
question might be interesting for further research.

\begin{prob}\label{prob:genGPG}
Characterise game-perfect graphs for the games $[B,A]$ and $[A,A]$ (by
forbidden induced subgraphs and/or explicit structural descriptions).
\end{prob} 

\rd{Note that we have no idea how to extend our methods to the more general case
of Problem~\ref{prob:genGPG}.
This might be very difficult as there are infinitely many minimal
forbidden configurations, namely (among others) all odd antiholes
\cite[Thm~23]{andresgperfect}.}

\rd{There is a historic analog for this discrepancy: 
the characterisation of line-perfect graphs by forbidden subgraphs
was found by \cite{trotter} in 1977, 
but the more general result,
the characterisation of perfect graphs by forbidden induced subgraphs
(the famous Strong Perfect Graph Theorem)
was proved by \cite{spgt}
nearly 30 years later and published in 2006.}

\rd{We remark that an analog for the explicit characterisation of 
line perfect graphs by \cite{maffray} has
not yet been found (more than 30 years later)
for perfect graphs.}

\subsection{\rd{A variant of line game-perfectness}}

One might also consider a variant of line game-perfectness: A graph $G$
is \emph{edge $[X,Y]$-perfect} if, for any edge-induced subgraph $H$ of $G$,
\[\chi_{[X,Y]}'(H)=\Delta(H).\]
By Corollary~\ref{cor:whitney}, the only difference between line $[X,Y]$-perfect
graphs and edge $[X,Y]$-perfect graphs is that in edge $[X,Y]$-perfect graphs
we have an additional forbidden configuration, namely the triangle $K_3$. Thus,
edge $[X,Y]$-perfect graphs can be obtained from our explicit structural
descriptions of line $[X,Y]$-perfect graphs by deleting all graphs that contain
a triangle, which leaves fairly trivial classes of graphs. Therefore our
notion of line $[X,Y]$-perfect graphs might be the better concept to describe
game-perfectness for edge colouring games.

\subsection{\rd{Games with \rdrd{a} bounded number of skipping turns}}
\label{subsec:reviewer}

\rd{In our games, skipping a turn was either forbidden or allowed for an
\rdrd{unlimited} number of turns. Now, we consider the question what happens
if we allow only a bounded number of skipping turns.}

\rd{Let $X,Y\in\{A,B\}$, and $k,s\in{\mathbb{N}}$, and $G$ be a graph. 
In the \emph{edge colouring game} $[X,Y]_s$ 
played with $k$ colours on the graph $G$ the
players alternately move with player $x$ beginning. Player $Y$ may skip
a turn (including the first one) up to $s$ times. A move that is not skipped
consists in colouring an uncoloured edge $e$ of $G$ with a colour from the
set $\{1,2,\ldots,k\}$ that \rdrd{is} different from the colours of the edges adjacent
to~$e$. This game defines a \emph{game chromatic index} $\chi_{[X,Y]_s}'(G)$
of the graph~$G$. The graph $G$ is \emph{line $[X,Y]_s$-perfect} if,
for any edge-induced subgraph $H$ of $G$,}
\[\omega(L(H))=\chi_{[X,Y]_s}'(H).\]
\rd{The class of all line $[X,Y]_s$-perfect graphs is denoted by
${\cal LP}[X,Y]_s$. By definition,} 
\[{\cal LP}[X,Y]_0={\cal LP}[X,-].\]

\rd{We observe the following.}

\begin{obs}\label{obs:skipone}
\rd{Let $X\in\{A,B\}$. Then we have:
\begin{itemize}
\item[(i)]
${\cal LP}[X,B]\subseteq\ldots\subseteq
{\cal LP}[X,B]_3\subseteq
{\cal LP}[X,B]_2\subseteq
{\cal LP}[X,B]_1\subseteq
{\cal LP}[X,-]$
\item[(ii)]
${\cal LP}[X,-]\subseteq
{\cal LP}[X,A]_1\subseteq
{\cal LP}[X,A]_2\subseteq
{\cal LP}[X,A]_3\subseteq\ldots\subseteq
{\cal LP}[X,A]$
\end{itemize}}
\end{obs}

\begin{proof}
\rd{This holds since the possibility to skip one time more is no disadvantage
for the player who is allowed to skip.}
\end{proof}

\begin{obs}\label{obs:skiptwo}
\rd{Let $s\in{\mathbb{N}}$. Then we have:
\begin{itemize}
\item[(i)] ${\cal LP}[B,B]_{s+1}\subseteq{\cal LP}[A,B]_{s}$
\item[(i)] ${\cal LP}[B,A]_{s}\subseteq{\cal LP}[A,A]_{s+1}$
\end{itemize}}
\end{obs}

\begin{proof}
\rd{Ad (i): If Bob has a winning strategy for the game $[A,B]_s$ on a
graph $G$, then, by skipping his first turn, he can use the same
strategy in order to win the game $[B,B]_{s+1}$ played on~$G$.} 

\rd{Ad (ii): If Alice has a winning strategy for the game $[B,A]_s$ on a
graph $G$, then, by skipping her first turn, she can use the same
strategy in order to win the game $[A,A]_{s+1}$ played on~$G$.} 
\end{proof}

\rd{Using the two observations above, we obtain the following corollary
of Theorem~\ref{thm:lineBB}.}

\begin{cor}\label{cor:skip}
\rd{Let $s\in{\mathbb{N}}$ with $s\ge1$. Then}
\[{\cal LP}[X,B]_s={\cal LP}[B,B].\]
\end{cor}

\begin{proof}
\rd{For any $s\in{\mathbb{N}}$ with $s\ge1$, by Theorem~\ref{thm:lineBB},
Observation~\ref{obs:compareclasses},
Observation~\ref{obs:skipone}~(i), and Observation~\ref{obs:skiptwo}~(i), 
we have}
\begin{eqnarray*}
{\cal LP}[B,B]
\stackrel{\text{Obs}~\ref{obs:compareclasses}}{\subseteq}{\cal LP}[A,B]
&\stackrel{\text{Obs}~\ref{obs:skipone}~(i)}{\subseteq}&{\cal LP}[A,B]_s\\
&\stackrel{\text{Obs}~\ref{obs:skipone}~(i)}{\subseteq}&{\cal LP}[A,-]
\stackrel{\text{Thm}~\ref{thm:lineBB}}{=}{\cal LP}[B,B]
\end{eqnarray*}
\rd{and}
\begin{eqnarray*}
{\cal LP}[B,B]
\stackrel{\text{Obs}~\ref{obs:skipone}~(i)}{\subseteq}{\cal LP}[B,B]_s
&\stackrel{\text{Obs}~\ref{obs:skiptwo}~(i)}{\subseteq}&{\cal LP}[A,B]_{s-1}\\
&\stackrel{\text{Obs}~\ref{obs:skipone}~(i)}{\subseteq}&{\cal LP}[A,-]
\stackrel{\text{Thm}~\ref{thm:lineBB}}{=}{\cal LP}[B,B].
\end{eqnarray*}
\rd{Thus, the equality ${\cal LP}[X,B]_s={\cal LP}[B,B]$
is true when $X=A$ or $X=B$.}
\end{proof}

\rd{According to Corollary~\ref{cor:skip}, our new games give no new classes
in case the skipping is allowed to Bob. The situation changes if the
skipping is allowed to Alice. Here the characterisation of the 
respective classes
of line game-perfect graphs seems to be very intricate. One reason is that
in our strategies for Alice sometimes Alice has to miss a turn in order
to avoid beginning to colour in a new, uncoloured component, which
makes the discussion of disconnected graphs very difficult. But
even for connected graphs our strategies require that Alice skips
several times. This is the case for single or double galaxies, where
our strategies require that Alice skips if Bob plays on the second
star edge of a pending \rdrd{triangle. Note that} a (single or double) galaxy may
have arbitrarily many pending \rdrd{triangles; therefore it might}
seem to be straightforward
that some of the classes ${\cal LP}[X,A]_s$ are different from
the classes ${\cal LP}[X,-]$ and ${\cal LP}[X,A]$.
However, it is not clear whether the strategies given in this
paper cannot be improved in some way using fewer skipping moves.}

\begin{prob}\label{prob:reviewer}
\rd{For any $s\in{\mathbb{N}}\setminus\{0\}$ and $X\in\{A,B\}$, 
characterise the class ${\cal LP}[X,A]_s$, i.e.,
characterise line game-perfect graphs for the game $[X,A]_s$ (by
forbidden edge-induced subgraphs and/or explicit structural descriptions).}
\end{prob}

\rd{The edge-colouring games $[X,Y]_s$ defined in this section might be considered
more generally, thus, we might define vertex colouring games $[X,Y]_s$
in the same way. Then we might ask the following question.}

\begin{prob}
\rd{For any $s\in{\mathbb{N}}\setminus\{0\}$ and $X,Y\in\{A,B\}$, 
characterise game-perfect graphs for the game $[X,Y]_s$ (by
forbidden induced subgraphs and/or explicit structural descriptions).}
\end{prob}

\rd{We expect that the answer to this question will be even more intricate
than the answer to Problem~\ref{prob:reviewer}.}


\acknowledgements
\label{sec:ack}
\rd{The authors thank the two \rdrd{anonymous} reviewers for many useful suggestions
that helped to improve the presentation of the paper.
Furthermore, in particular, we acknowledge \rdrd{that} the idea of the games 
discussed in Section~\ref{subsec:reviewer}
\rdrd{originates} from one of the reviewers.}

\nocite{*}
\bibliographystyle{abbrvnat}
\bibliography{linegameperfect}

\begin{thebibliography}{27}
\providecommand{\natexlab}[1]{#1}
\providecommand{\url}[1]{\texttt{#1}}
\expandafter\ifx\csname urlstyle\endcsname\relax
  \providecommand{\doi}[1]{doi: #1}\else
  \providecommand{\doi}{doi: \begingroup \urlstyle{rm}\Url}\fi

\bibitem[Andres(2006)]{andresforest}
S.~D. Andres.
\newblock The game chromatic index of forests of maximum degree {$\Delta\ge5$}.
\newblock \emph{Discrete Applied Mathematics}, 154:\penalty0 1317--1323, 2006.

\bibitem[Andres(2007)]{andresphd}
S.~D. Andres.
\newblock \emph{Digraph coloring games and game-perfectness}.
\newblock Verlag Dr.\ Hut, {M\"{u}nchen}, 2007.
\newblock Ph.D.\ thesis, {Universit\"{a}t} zu {K\"{o}ln}.

\bibitem[Andres(2009)]{andresgperfect}
S.~D. Andres.
\newblock Game-perfect graphs.
\newblock \emph{Mathematical Methods of Operations Research}, 69:\penalty0
  235--250, 2009.

\bibitem[Andres(2012)]{char}
S.~D. Andres.
\newblock On characterizing game-perfect graphs by forbidden induced subgraphs.
\newblock \emph{Contributions to Discrete Mathematics}, 7:\penalty0 21--34,
  2012.

\bibitem[Andres and Lock(2019)]{andreslock}
S.~D. Andres and E.~Lock.
\newblock Characterising and recognising game-perfect graphs.
\newblock \emph{Discrete Mathematics \& Theoretical Computer Science},
  21\penalty0 ({Paper No. 6}):\penalty0 39pp., 2019.

\bibitem[Andres et~al.(2011)Andres, Hochst\"{a}ttler, and
  Schall\"{u}ck]{andresetalwheel}
S.~D. Andres, W.~Hochst\"{a}ttler, and C.~Schall\"{u}ck.
\newblock The game chromatic index of wheels.
\newblock \emph{Discrete Applied Mathematics}, 159:\penalty0 1660--1665, 2011.

\bibitem[Bartnicki and Grytczuk(2008)]{bartnickigrytczuk}
T.~Bartnicki and J.~Grytczuk.
\newblock A note on the game chromatic index of graphs.
\newblock \emph{Graphs and Combinatorics}, 24:\penalty0 67--70, 2008.

\bibitem[Beineke(1970)]{beineke}
L.~W. Beineke.
\newblock Characterizations of derived graphs.
\newblock \emph{Journal of Combinatorial Theory}, 9:\penalty0 129--135, 1970.

\bibitem[Beveridge et~al.(2008)Beveridge, Bohman, Frieze, and
  Pikhurko]{beveridgeetal}
A.~Beveridge, T.~Bohman, A.~Frieze, and O.~Pikhurko.
\newblock Game chromatic index of graphs with given restrictions on degrees.
\newblock \emph{Theoretical Computer Science}, 407:\penalty0 242--249, 2008.

\bibitem[Bodlaender(1991)]{bodlaender}
H.~L. Bodlaender.
\newblock On the complexity of some coloring games.
\newblock \emph{International Journal of Foundations of Computer Science},
  2:\penalty0 133--147, 1991.

\bibitem[Boudon et~al.(2017)Boudon, Przyby{\l}o, Senhaji, Sidorowicz, Sopena,
  and Wo\'{z}niak]{boudonetal}
O.~Boudon, J.~Przyby{\l}o, M.~Senhaji, E.~Sidorowicz, E.~Sopena, and
  M.~Wo\'{z}niak.
\newblock The neighbour-sum-distinguishing edge-colouring game.
\newblock \emph{Discrete Mathematics}, 340:\penalty0 1564--1572, 2017.

\bibitem[Cai and Zhu(2001)]{caizhu}
L.~Cai and X.~Zhu.
\newblock Game chromatic index of $k$-degenerate graphs.
\newblock \emph{Journal of Graph Theory}, 36:\penalty0 144--155, 2001.

\bibitem[Chan and Nong(2014)]{channong}
W.~H. Chan and G.~Nong.
\newblock The game chromatic index of some trees of maximum degree 4.
\newblock \emph{Discrete Applied Mathematics}, 170:\penalty0 1--6, 2014.

\bibitem[Charpentier et~al.(2018)Charpentier, Effantin, and
  Paris]{charpentierep}
C.~Charpentier, B.~Effantin, and G.~Paris.
\newblock On the game coloring index of $f^+$-decomposable graphs.
\newblock \emph{Discrete Applied Mathematics}, 236:\penalty0 73--83, 2018.

\bibitem[Chudnovsky et~al.(2006)Chudnovsky, Robertson, Seymour, and
  Thomas]{spgt}
M.~Chudnovsky, N.~Robertson, P.~Seymour, and R.~Thomas.
\newblock The strong perfect graph theorem.
\newblock \emph{Annals of Mathematics (2)}, 164:\penalty0 51--229, 2006.

\bibitem[Dunn(2007)]{dunnrelkdeg}
C.~Dunn.
\newblock The relaxed game chromatic index of $k$-degenerate graphs.
\newblock \emph{Discrete Mathematics}, 307:\penalty0 1767--1775, 2007.

\bibitem[Dunn et~al.(2015)Dunn, Morawski, and Nordstrom]{dunnmorawskinordstrom}
C.~Dunn, D.~Morawski, and J.~F. Nordstrom.
\newblock The relaxed edge-coloring game and k-degenerate graphs.
\newblock \emph{Order}, 32:\penalty0 347--361, 2015.

\bibitem[Fong and Chan(2019{\natexlab{a}})]{fongchan}
W.~L. Fong and W.~H. Chan.
\newblock The edge coloring game on trees with the number of colors greater
  than the game chromatic index.
\newblock \emph{Journal of Combinatorial Optimization}, 38:\penalty0 456--480,
  2019{\natexlab{a}}.

\bibitem[Fong and Chan(2019{\natexlab{b}})]{fongchanarxiv}
W.~L. Fong and W.~H. Chan.
\newblock The game chromatic index of some trees with maximum degree 4 with at
  most three degree-four vertices in a row.
\newblock \emph{arXiv}, 2019{\natexlab{b}}.
\newblock arXiv:1904.01496.

\bibitem[Fong et~al.(2018)Fong, Chan, and Nong]{fongchannong}
W.~L. Fong, W.~H. Chan, and G.~Nong.
\newblock The game chromatic index of some trees with maximum degree four and
  adjacent degree-four vertices.
\newblock \emph{Journal of Combinatorial Optimization}, 36:\penalty0 1--12,
  2018.

\bibitem[Gardner(1981)]{gardner}
M.~Gardner.
\newblock Mathematical games.
\newblock \emph{Scientific American}, 244:\penalty0 23--26, April 1981.

\bibitem[Keusch(2018)]{keusch}
R.~Keusch.
\newblock A new upper bound on the game chromatic index of graphs.
\newblock \emph{Electronic Journal of Combinatorics}, 25\penalty0 ({Paper No.
  2.33}):\penalty0 18pp., 2018.

\bibitem[Lam et~al.(1999)Lam, Shiu, and Xu]{lamshiuxu}
P.~C.~B. Lam, W.~C. Shiu, and B.~Xu.
\newblock Edge game-coloring of graphs.
\newblock \emph{Graph Theory Notes New York}, 37:\penalty0 17--19, 1999.

\bibitem[Lock(2016)]{lock}
E.~Lock.
\newblock \emph{The structure of $g_B$-perfect graphs}.
\newblock {FernUniversit\"{a}t in Hagen}, 2016.
\newblock Bachelor thesis.

\bibitem[Maffray(1992)]{maffray}
F.~Maffray.
\newblock Kernels in perfect line-graphs.
\newblock \emph{Journal of Combinatorial Theory B}, 55:\penalty0 1--8, 1992.

\bibitem[Trotter(1977)]{trotter}
L.~E. Trotter.
\newblock Line perfect graphs.
\newblock \emph{Mathematical Programming}, 12:\penalty0 255--259, 1977.

\bibitem[Whitney(1932)]{whitney}
H.~Whitney.
\newblock Congruent graphs and the connectivity of graphs.
\newblock \emph{American Journal of Mathematics}, 54:\penalty0 150--168, 1932.

\end{thebibliography}
\label{sec:biblio}

\end{document}